\documentclass[11pt]{article}
\usepackage[margin=1in]{geometry}
\usepackage{graphicx} 
\usepackage{subcaption}
\usepackage{amsmath}
\usepackage{tikz}
\usepackage{float}
\usetikzlibrary{trees}
\usepackage{array}
\usepackage{amsfonts}
\usepackage{mathtools}

\usepackage{fullpage}
\usepackage{stmaryrd}
\usepackage{soul}
\usepackage{color}
\usepackage{times}
\usepackage{fancyhdr,graphicx,amsmath,amssymb}
\usepackage{amsthm}
\usepackage{bbm}
\usepackage{dsfont}
\usepackage{makecell}
\usepackage{hyperref} 
\usepackage{cleveref} 

\usepackage[ruled,vlined]{algorithm2e}
\newtheorem{lemma}{Lemma}[section]

\newtheorem{definition}[lemma]{Definition}

\newtheorem{theorem}[lemma]{Theorem}
\newtheorem{remark}[lemma]{Remark}
\newtheorem{proposition}[lemma]{Proposition}

\newtheorem{corollary}[lemma]{Corollary}

\newtheorem{claim}[lemma]{Claim}

\DeclareMathOperator{\argmax}{argmax}
\DeclareMathOperator{\argmin}{argmin}

\DeclareMathOperator{\supp}{supp}

\DeclareMathOperator{\graph}{graph}
\DeclareMathOperator{\conv}{conv}

\newenvironment{sproof}{%
  \proof}{\endproof}

\newtheorem{assumption}{Assumption}

\newcommand{\red}[1]{{\color{black} #1}}

\newcommand{\R}{\mathbb R}

\newcommand{\RR}{\mathbb{R}}

\newcommand{\prox}{\mathrm{prox}}

\usepackage{capt-of}
\usepackage{bbm}
\usepackage{hyperref}
\numberwithin{equation}{section}
\begin{document}
\title{Statistical Estimation of Monge Transport Maps via Brenier Potentials}
\author{
  Elsa Cazelles\\
  CNRS, IRIT, Université de Toulouse
  \and
  Edouard Pauwels\\
  Toulouse School of Economics\\ Université Toulouse Capitole
  \and
  Léo Portales\\
  IRIT, TSE, INP Toulouse
}
\date{\today}
\maketitle

\begin{abstract}
We introduce and analyze a statistical estimator for Monge transport maps: solutions to the quadratic optimal transport problem in the Euclidean space. For absolutely continuous source measures, this map is uniquely defined as the gradient of a convex function, a result known as Brenier’s theorem. Without absolute continuity, the problem is relaxed, maps are replaced by coupling measures, and optimal couplings are supported on the subdifferential of a convex function, called a Brenier potential. This potential is the basis for our transport map statistical estimator, for measures known only through finite samples. We construct an explicit optimal Brenier potential for the empirical problem, with a simple closed form expression based on dual solutions. Our transport map estimator is then defined as a suitable subdifferential selection, and can be extended to $\R^d$.
We exhibit convergence rates for this estimator based on a new error bound for the quadratic optimal transport problem. In the semi-discrete setting, where the target measure is finitely supported, our estimator enjoys sharper convergence rates. Our methodology does not rely on smoothness or continuity of the target Monge transport map and requires no computation beyond primal-dual solutions of the empirical finite dimensional problem. Finally, using similar proof techniques, we provide novel convergence rate for empirical couplings.
\end{abstract}

\section{Introduction}

Optimal transport, as initially formulated by Monge, seeks a map which ``pushes-forward'' a source probability measure $\mu$ onto a target probability measure $\nu$ while minimizing the quadratic cost. The classical Monge problem is written over $\R^d$ as follows, over the class of measurable mappings $T$
\begin{equation}
    \label{eq:preliminary_Monge}
    \underset{T_\sharp\mu=\nu}{\inf_{T\colon\R^d\to\R^d}}\int_{\R^d}\frac{1}{2}\|x-T(x)\|^2d\mu(x),
\end{equation}
where $\sharp$ denotes the push forward operation. The constraint can be equivalently formulated as $\int f(y) d\nu(y) = \int f(T(x)) d \mu(x)$ for all measurable $f$ on $\R^d$.
A solution to this problem, when it exists, will be called a \textit{Monge map}. 

Optimal transport is one of the key mathematical problems of the late 20-th century, both as an analytical tool and as a mathematical modeling framework with many applications \cite{villani_opt_old_new,book_santambrogio,peyre2019computational}. We limit ourselves to the quadratic case, as in \eqref{eq:preliminary_Monge}. Our question of interest is the statistical estimation of $T$ for source and target $\mu,\nu$ known only through independent samples, or equivalently, based on discrete empirical measures $\hat{\mu}_n,\hat{\nu}_m$. We introduce a general estimator for $T$, and provide finite sample guaranties under mild assumptions on problem data: compact support of the source and target measures, and absolute continuity of the source measure.

Monge problem admits a (semi)dual formulation over the set of convex functions on $\R^d$ as follows,
\begin{align}
    \label{eq:preliminary_semi_dual}
    \inf_{\varphi\in \Gamma_0(\R^d)} \int_{\R^d} \varphi(x) d\mu(x) + \int_{\R^d}\varphi^*(y) d\nu(y),
\end{align}
where $\Gamma_0(\R^d)$ denotes the set of convex lower semicontinuous proper functions  over $\R^d$ and $\varphi^* \colon y \mapsto \sup_{x \in \R^d} \left\langle x, y \right\rangle - \varphi(x)$ is the Legendre-Fenchel conjugate. A celebrated result of Brenier \cite{brenier1991polar} ensures that for an absolutely continous source measure $\mu$, the solution to \eqref{eq:preliminary_Monge} exists, is unique and is the (sub)gradient of a convex function $\varphi \in \Gamma_0(\R^d)$. This function $\varphi$ turns out to be an optimal solution to problem \eqref{eq:preliminary_semi_dual} \cite{book_santambrogio}. Solutions to \eqref{eq:preliminary_semi_dual} will be called Brenier potentials. 

These elements are well known \cite{villani_opt_old_new,ref_villani_topics, book_santambrogio}, and we provide a complete description in \Cref{sec:preliminary}. We also expand on the discrete optimal transport problem, between finitely supported measures, with the corresponding finite dimensional linear programming formulation \cite{peyre2019computational}.

\subsection{Overview of the main contributions}

\paragraph*{An explicit  Brenier potential} For discrete measures $\mu_n$ and $\nu_m$, supported on $n$ and $m$ points respectively, we show that the dual problem \eqref{eq:preliminary_semi_dual} admits an explicit solution which takes the form of the maximum of $m$ affine functions.  The coefficients of these affine functions are directly deduced from the support points of the target $\nu_m$ and the dual solution of the discrete optimal transport linear program, see Section \ref{sec:estimator} and in particular Definition \ref{def:mainEstimator}. This Brenier potential has basically the same computational cost as computing the discrete optimal transport problem, as most solvers provide primal-dual solutions. Its structure is reminiscent of the Laguerre tesselation occurring in semi-discrete optimal transport \cite{aurenhammer1998minkowski}, but in a discrete-discrete setting. This construction serves as a basis for our statistical estimator.

\paragraph*{Finite sample guaranties} In a statistical setting, given population source (which is absolutely continuous) and target probability measures $\mu,\nu$, and empirical measures $\hat{\mu}_n$, $\hat{\nu}_m$ supported on $n$ and $m$ points respectively, we have the following under mild assumptions.
\begin{itemize}
    \item As known, the population problem \eqref{eq:preliminary_Monge} for $(\mu,\nu)$ admits a solution $T = \nabla \bar{\varphi}$, where $\bar{\varphi}$ solves \eqref{eq:preliminary_semi_dual}.
    \item For the discrete random problem between $\hat{\mu}_n$ and $\hat{\nu}_m$, we have an explicit Brenier  map estimator $\hat{\phi}$, solution to \eqref{eq:preliminary_semi_dual} as described in the previous paragraph.
\end{itemize}
It is therefore natural to consider $\nabla \hat{\phi}$ as a statistical estimator for $\nabla \bar{\varphi}$. This estimator enjoys computational simplicity (closed form expression), without any parameter to estimate or tune. 

The main statistical question to be addressed is that of the convergence of $\nabla \hat{\phi}$ as the sample sizes $n$ and $m$ increase. Noting that the convex subdifferential (see \cite{rockafellar1970convex}) satisfies $\partial \hat{\phi}(x) = \{\nabla \hat{\phi}(x)\}$ for almost all $x$, it is possible to invoke the main results of \cite{segers2022graphical} to obtain that, almost surely, $\partial \hat{\phi}$ converges graphically to $\partial \bar{\varphi}$, restricted to the interior of the support of $\mu$. This is a qualitative result, illustrating the fact that $\nabla \hat{\phi}$ is a relevant estimator. However, the contribution of \cite{segers2022graphical} remains qualitative and the underlying notion of convergence is difficult to explicit in functional analysis terms. We considerably extend this result for our specific estimator and provide a finite sample upper bound on the quantity $\|\nabla \hat{\phi} - \nabla \bar{\varphi}\|_{L^1(\mu)}$. This upper bound holds under mild assumptions: compact support of $\mu,\nu$, and upper bounded density of $\mu$. In particular we make no assumption regarding the convexity of the support of $\mu$ or any lower bound on its density, and the population Monge map $\nabla \bar{\varphi}$ may be discontinuous. This is our main statistical result, it is given in \Cref{sec:mainStatisticalResult}. A visualization of the estimator and its convergence is given in \Cref{fig:illustration}.

\begin{figure}
    \centering
    \includegraphics[width=1.0\linewidth]{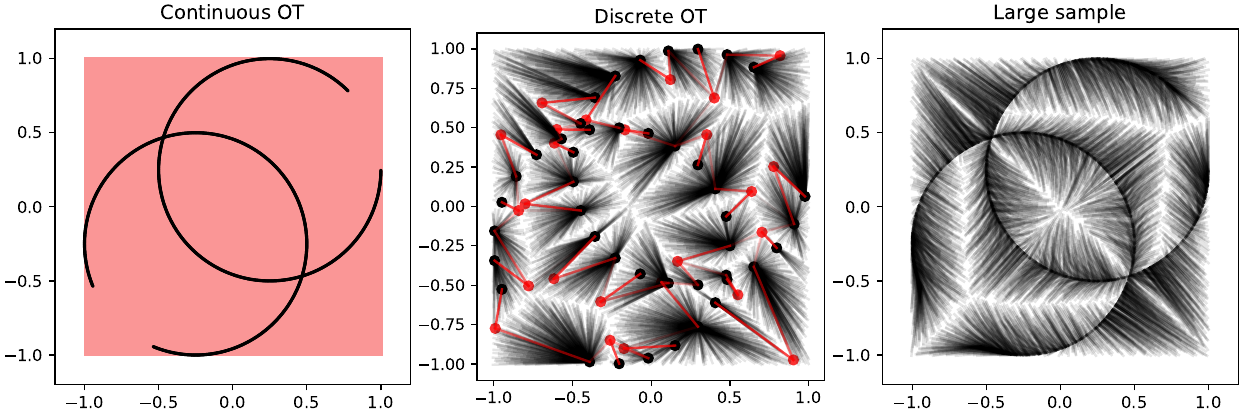}
    \caption{Visualization of the main contribution in the plane. \textit{Left:} Population problem, the goal is to send the uniform red source density on the square to the black singular target measure. \textit{Center:} Discrete, sampled, version of the problem with empirical measures. Dots represent sample from source (red) and target (black). The red lines represent the discrete optimal transport solutions (red mass is split as there are more black points). The black lines represent the Monge map estimator, computed from dual optimal solutions of the discrete problem. It is piecewise constant, defined on the whole plane, and each black line represents the association between a new source input and its image.  \textit{Right:} For large samples, the estimator converges to the population Monge map, illustrating our finite sample $L^1$ bound. Note that, in this case, the population Monge Map is not continuous and our results apply to this setting. Discrete solutions are computed using POT \cite{flamary2021pot}.
    }
    \label{fig:illustration}
\end{figure}

\paragraph*{A dual error bound} The key analytical element which allows to devise our finite sample estimate is an error bound for Problem \eqref{eq:preliminary_semi_dual}. We show in \Cref{sec:errorBound} that for compactly supported probability measures $\mu,\nu$, the source $\mu$ with an upper bounded density and $\bar{\varphi}$ solution to \eqref{eq:preliminary_semi_dual}, there is $c>0$ such that for any convex function $\phi$ in an appropriate subclass of $\Gamma_0(\RR^d)$, it holds that
\begin{align}
    c\|\nabla \phi - \nabla \bar{\varphi}\|_{L^1(\mu)}^4 \leq  \int \phi - \bar\varphi d\mu + \int \phi^* - \bar\varphi^* d\nu.
    \label{eq:prelimErrorBound}
\end{align}
Choosing $\phi = \hat{\phi}$,  our Brenier map estimator from empirical measures $\hat{\mu}_n,\hat{\nu}_m$, the right hand side can be further upper bounded by $W_1(\mu,\hat{\mu}_n) + W_1(\nu,\hat{\nu}_m)$ for which finite sample upper bounds exist in expectation in \cite{fournier2015rate}, see \cite{fournier2023convergence} for explicit multiplicative constants. This results in statistical rates roughly of order $O(n^{-\frac{1}{4d}} + m^{-\frac{1}{4d}})$. We do not know if these rates could be improved for the proposed estimator, or for other estimators. The closest lower bound is $n^{-\frac{1}{d}}$, but under  considerably more restrictive conditions on the measures, see Section \ref{sec:comparisonStat}.

The error bound is of independent interest for quadratic optimal transport beyond statistical estimation. It relates for example to stability properties for Monge maps \cite{gigli2011holder} with applications in machine learning \cite{merigot2020quantitative}. The proof combines Carlier's remainder for Fenchel-Young inequality \cite{carlier2023fenchel} with the covering number estimate from \cite{carlier2025quantitative}, which we use as a measure of discontinuity for gradients of Lipschitz convex functions. 

\paragraph*{Extensions} The proposed framework and techniques are extended to provide refined statistical bounds for the semi-discrete optimal transport problem in \Cref{sec:semi-discr}. In \Cref{sec:couplings}, quantitative convergence guaranties of empirical optimal couplings toward optimal couplings are given, based on quantitative stability bounds in \Cref{sec:stabilityCouplings}. The consistency of empirical coupling is due to \cite{segers2022graphical}.

\subsection{Related work}

\paragraph*{Statistical optimal transport}
Research on the topic has considerably expanded in the last decade. This includes many variations: semi-discrete, entropic regularization, convergence of the optimal transport cost, of the optimal map, etc. We focus on the estimation of quadratic optimal transport Monge maps. We refer to \cite{balakrishnan2025statistical} for a recent review and will comment here on the specificity of our approach and the closest references.

The idea of estimating Monge maps based on the dual optimal transport formulation was explored in \cite{hutter2021minimax} and expanded in \cite{manole2024plugin,divol2025optimal}. This literature deploys the tools of non parametric statistical estimation to tackle the problem. The main drawback of this approach is that it relies on smoothness  of the population transport map, which is required to be an element of a well identified function space, such as that of bi-Lipschitz diffeomorphisms. Regularity of optimal transport is a topic of notable mathematical difficulty and sufficient conditions impose considerable restrictions. The main specificity of our work with respect to this literature is that we do not directly apply nonparametric estimation to $\nabla \varphi$, but rather leverage a new error bound in \eqref{eq:prelimErrorBound}. This allows to shift the regularity requirements from the Monge map $\nabla \varphi$ to the Brenier potential $\varphi$ and leverage known convergence rates for empirical measures \cite{fournier2015rate}. As a consequence our main result holds under explicit and versatile assumptions on the problem: compacity of the supports, absolute continuity of the source, and an upper bound on its density. A striking consequence is that the target Monge map is not even required to be continuous. The only available occurrence of statistical convergence toward possibly discontinuous Monge maps is \cite{pooladian2023minimax} which is limited to the semi-discrete setting. A more detailed discussion on these aspects is given in \Cref{sec:comparisonStat}.

\paragraph*{Estimators for Monge maps} A diversity of Monge map estimators have been proposed in the literature with various statistical analyses. These include M-estimators \cite{hutter2021minimax,manole2024plugin,divol2025optimal,chewi2025estimation}, entropic regularization \cite{seguy2018large,pooladian2021entropic}, kernel based estimators \cite{perrot2016mapping,gunsilius2022convergence}, input convex neural networks \cite{makkuva2020optimal}, one nearest neighbors and barycentric projection \cite{manole2024plugin}, convex extrapolation and smoothing \cite{paty2020regularity,hallin2021distribution,de2021consistent}, kernel sums-of-squares \cite{vacher2021dimension,vacher2024optimal}.  

Interestingly, our proposed Brenier potential has not been considered explicitly. It is a natural estimator as it is the solution of Problem \eqref{eq:preliminary_semi_dual} with discrete measures and it admits a closed form solution based on the optimal transport linear program dual solutions. As discussed in \Cref{rem:entropicRegularizationLimit}, our estimator corresponds to the limit as entropic regularization vanishes of the estimator proposed in \cite{seguy2018large} and analyzed in \cite{pooladian2021entropic,pooladian2023minimax}. Finally, \cite{vacher2022parameter} explores the distance between gradients of convex potentials and the dual loss, under smoothness condition on the source density, in a statistical setting.

\paragraph*{Stability of optimal transport maps}
The error bound in \eqref{eq:prelimErrorBound} translates into a H\"older stability bound for Monge maps. This topic is of interest more broadly for optimal transport 
\cite{gigli2011holder,merigot2020quantitative,li2021quantitative,delalande2023quantitative,letrouit2024gluing}. Our error bound directly implies a similar stability bound, allowing for variations of both source and target probability measures. After finalizing the first version of this manuscript, we became aware of the independent work \cite{merigot2026sharp}, providing improved exponents for this problem using techniques similar to ours, see also the discussion in the conclusion. These exponents turn out to be optimal as shown by \cite{letrouit2025unstable}. This work also provides an error bound similar to \eqref{eq:prelimErrorBound}.
A more precise discussion of the relation to existing art is given in \Cref{sec:connection_literature_stability} and in the conclusion.

\paragraph*{Semi-discrete optimal transport} Our estimator is strongly related to Laguerre tesselation, adapted to a discrete setting, a notion which has very strong ties with semi-discrete optimal transport \cite{aurenhammer1987power}. Similar piecewise linear estimators are studied in convex geometry \cite{gu2013variational}, in relation to optimal transport and its applications \cite{lei2019geometric}.

Finally, we specify our results to the semi-discrete setting, that is for $\nu$ a discrete probability measure, showing that it is possible to remove the exponential dependency in the ambient dimension $d$. This is in line with existing statistical results on this topic \cite{pooladian2023minimax,sadhu2024stability}. Statistical aspects for the semi-discrete setting are further discussed in \Cref{sec:semi-discr}.

\subsection{Notation}
Throughout the paper $\|.\|$ will denote the standard Euclidean norm. For any $R>0$, $B_R:=B_d(0,R)$ and $\bar{B}_R:=\bar{B}_d(0,R)$ will denote respectively the $d$-dimensional Euclidean open and closed ball centered in $0$ with radius $R$. For any function $f$ and measure $\mu$, we denote $\graph(f)$ the graph of $f$ and $\supp \mu$ the support of $\mu.$ 
For any function $f:\R^d\rightarrow\R$ we will denote $f^*$ its Fenchel-Legendre transform, that is, for all $y\in\R^d:\:f^*(y)=\sup_{x\in\R^d}\langle x,y \rangle-f(x)$. For any set of points $\{y_i\}_{i\in I}$ we denote $\conv\{y_i,\:i\in I\}$ its convex hull.
We also define $\Gamma_0(\R^d)$, the classical set of proper lower semi-continuous convex functions on $\R^d$, that is:
\begin{equation*}
\Gamma_0(\R^d)=\left\{ \phi:\R^d\longrightarrow\R\cup \{+\infty \},\:\phi\: \text{is convex, l.s.c, proper } \right\}.
\end{equation*}
We call the domain of a function $f$ the set of points where $f$ is finite. For any set $C$, we denote $\delta(.|C)$ the indicator function which values $0$ at $C$ and $+\infty$ elsewhere.
For measure subsets, we introduce the following notation. For any subset $X$ of $\R^d$, we denote $\mathcal{M}_1(X)$ the set of probability measures with support included in $X$. Finally for any $(\mu,\nu)\in\mathcal{M}_1(X)\times\mathcal{M}_1(Y)$ we will denote the set of probability measures with respective marginals $\mu$ and $\nu$ as follows
\begin{align*}
\Pi(\mu,\nu)=\big\{\pi\in\mathcal{M}_1(X\times Y)\:|&\:\pi(A\times Y)=\mu(A),\:\pi(X\times B)=\nu(B),\\
&\forall A\subset X,B\subset Y\:\text{measurable subsets}\big\}.
\end{align*}

For any mapping $f:\R^d\to \R^d$, $f_\sharp\mu$ denotes the push-forward of measure $\mu$ by $f$, that is, for any measurable subset $B\subset\R^d$, $f_\sharp\mu(B)=\mu(f^{-1}(B))$.

\section{Preliminaries on optimal transport}
\label{sec:preliminary}
We briefly recall basic concepts in optimal transport as well as important notions for the paper.

\subsection{Primal and dual problems}

The problem in \eqref{eq:preliminary_Monge} does not always have a solution. The Kantorovich relaxation replaces maps by probability couplings with marginals $\mu$ and $\nu$ (\textit{i.e} elements of $\Pi(\mu,\nu)$), as follows
\begin{equation}
    \label{eq:preliminary_Kantorovich}
    \inf_{\gamma\in\Pi(\mu,\nu)}\int_{\R^d\times \R^d}\frac{1}{2}\|x-y\|^2d\gamma(x,y).
\end{equation}
Duality is typically expressed in terms of c-concave functions \cite[Section 1.6.1]{book_santambrogio}. Focusing on the quadratic cost allows us to adopt a classical alternative: expressing duality with usual convexity notions. Both approaches are equivalent up to a change of variables \cite[Proposition 1.21]{book_santambrogio}. Problem \eqref{eq:preliminary_Kantorovich} admits the following dual
\begin{equation}\label{eq:preliminary_dual}
    \sup_{\varphi\in \Gamma_0(\R^d)}\int_{\R^d} \left(\frac{\|x\|^2}{2} - \varphi(x) \right)d\mu(x) + \int_{\R^d} \left(\frac{\|y\|^2}{2} - \varphi^*(y)\right) d\nu(y).
\end{equation}
This formulation is sometimes called "semi-dual", as it has a single decision variable while there are two marginal constraints in the primal. When $\mu$ is absolutely continuous, Brenier's Theorem tells us that the optimal Monge map $T$ in \eqref{eq:preliminary_Monge} is almost-everywhere the gradient of $\varphi \in \Gamma_0(\R^d)$, an optimal solution to \eqref{eq:preliminary_dual}, called a \textit{Brenier potential}. In the absence of continuity when the measures are compactly supported, optimal couplings $\gamma\in \Pi(\mu,\nu)$ exist for the Kantorovich problem \eqref{eq:preliminary_Kantorovich}. Moreover such solution $\gamma$ is supported on the graph of the subdifferential of $\varphi$, \textit{i.e.} $\supp \gamma \subset \graph(\partial \varphi)$, a solution to Problem \eqref{eq:preliminary_dual}. These are classical facts \cite{villani_opt_old_new,ref_villani_topics,book_santambrogio} which we  summarize in the following theorem. We provide a proof for the latter for completeness in \Cref{sec:proogTheoremClassicalDuality}

\begin{theorem}[Duality for quadratic optimal transport]
    Let $\mu,\:\nu$ be compactly supported probability measures. Consider the following problems 
    \begin{equation}\label{PP}\tag{PP}
        \inf_{\gamma\in\Pi(\mu,\nu)}\int_{\R^d \times \R^d}\frac{1}{2}\|x-y\|^2d\gamma(x,y)
    \end{equation}
    and
    \begin{equation}\label{DP}\tag{DP}
        \int_{\R^d} \frac{\|x\|^2}{2} d\mu(x)+ \int_{\R^d} \frac{\|y\|^2}{2}d\nu(y) - \inf_{\varphi\in \Gamma_0(\R^d)} \int_{\R^d} \varphi(x) d\mu(x) + \int_{\R^d}\varphi^*(y) d\nu(y).
    \end{equation}
    We have \eqref{PP} $=$ \eqref{DP} and optimal solutions to \eqref{PP} and \eqref{DP} exist. Moreover the following are equivalent, for any $\gamma \in \Pi(\mu,\nu)$ and $\varphi \in \Gamma_0(\R^d)$  
    \begin{enumerate}
        \item  $\gamma$ is optimal for \eqref{PP} and $\varphi$ is optimal for \eqref{DP}.
        \item $\int\frac{1}{2}\|x-y\|^2d\gamma(x,y) = \int \left(\frac{\|x\|^2}{2} - \varphi(x) \right)d\mu(x) + \int \left(\frac{\|y\|^2}{2} - \varphi^*(y)\right) d\nu(y)$
        \item $\supp \gamma\subset \graph(\partial \varphi)$.
    \end{enumerate}
    Furthermore, $\varphi$ may be restricted to be  $R_\nu$-Lipschitz with $R_\nu=\max_{y\in \supp \nu} \|y\|$.
    Finally, if $\mu$ is absolutely continuous, the solution to \eqref{PP} is unique and defined $\mu$-almost everywhere as $\gamma=(Id,\nabla\varphi)_\sharp\mu$, and $T = \nabla \varphi$ is solution to Problem \eqref{eq:preliminary_Monge}.
    \label{th:classical_duality_theorem}
\end{theorem}

In this article we will rely on an alternative formulation of this theorem, given below as a lemma.
\begin{lemma}[{Local Lipschitz continuity of Brenier potentials and their conjugate}]
    Let $\mu,\:\nu$ be as in \Cref{th:classical_duality_theorem} with compact support respectively included in  $B_{R_\mu}$ and $B_{R_\nu}$, then optimal Brenier potentials in \eqref{DP} may be chosen in
    \begin{equation*}
        \mathcal{C}_{R_\mu,R_\nu}=\left\{\varphi\in\Gamma_0(\R^d)\:\Big|\:\varphi\:R_\nu\text{-Lipschitz on}\;B_{R_\mu},\:\varphi^*\:R_\mu\text{-Lipschitz on }B_{R_\nu}\right\}.
    \end{equation*}
    In fact, for any $\varphi_0$ optimal in \Cref{th:classical_duality_theorem}, there exists $\varphi_1 \in \mathcal{C}_{R_\mu,R_\nu}$ such that $\varphi_1 = \varphi_0$ on $\supp \mu$ and $\varphi_1^* = \varphi_0^*$ on $\supp \nu$.
    \label{lem:brenier_potentials}
\end{lemma}

\begin{proof}
    Pick a $R_\nu$ Lipschitz optimal potential $\varphi_0$ in \Cref{th:classical_duality_theorem} and $\gamma$ an optimal coupling.
    We have $\supp \gamma \subset \graph \partial \varphi$ and $\supp \mu \subset B_{R_\mu}$. Furthermore $\supp \nu \subset \partial \varphi(B_{R_\mu})$ using Lemma \ref{lem:projection_support_coupling} \red{from the appendix} and noticing that $\{y\in\R^d,\:\exists x\in \supp \mu,\:(x,y)\in \supp \gamma \}\subset \partial \varphi(B_{R_\mu})$ using  the third point of Theorem \ref{th:classical_duality_theorem}.
    
    \Cref{lem:identification_lipschitz} \red{from the appendix} ensures that there is $\varphi_1 \in \mathcal{C}_{R_\mu,R_\nu}$ such that $\varphi_1 = \varphi_0$ on $\supp \ \mu$ and $\varphi_1^* = \varphi_0^*$ on $\supp \ \nu$. We deduce that $\int \varphi_1 d\mu = \int \varphi_0 d\mu$ and $\int \varphi_1^* d\nu = \int \varphi_0^* d\nu$ so that $\varphi_1$ is also optimal by point 2 of Theorem \ref{th:classical_duality_theorem}
\end{proof}

\begin{remark}
    Any $\varphi \in \Gamma_0(\R^d)$ Lipschitz over $B_{R_\mu}$ is also Lipschitz over $\overline{B}_{R_\mu}$. This is because Lipschitz continuity allows to extend $\varphi$ from the interior to the boundary, and this extension corresponds to the actual value of $\varphi$ by lower semi-continuity. 
    \label{rem:extension_lipschitz_border}
\end{remark}

\Cref{lem:brenier_potentials} is a regularity result on the solution of problem \eqref{DP}. Its specificity lies in the fact that both $\varphi$ and $\varphi^*$ are considered. Note that both cannot be globally Lipschitz at the same time. Let us mention that Lipschitz continuity of Kantorovich potentials was considered in \cite{mccann2001polar}, and that fine regularity is a crucial topic in optimal transport, starting with the work of Caffarelli \cite{caffarelli1992boundary,caffarelli1992regularity,caffarelli1996boundary}.

\subsection{Optimal transport between finitely supported measures}

Our basic framework for statistical estimation is for finitely supported source and target measures. Throughout this section, we consider the special case when $\hat\mu=\sum_{i=1}^na_i\delta_{x_i}$, $\hat\nu=\sum_{j=1}^mb_j\delta_{y_j}$, where $a_i> 0$ for $i=1, \ldots, n$, $b_j>0$ for $j=1,\ldots, m$, such that $\sum_{i=1}^na_i= \sum_{j=1}^mb_j=1$, and $x_1,\ldots,x_n$, $y_1,\ldots,\:y_m$ are atoms in $\R^d$. 
The optimal transport problem between such discrete measures falls within the scope of \Cref{th:classical_duality_theorem}, but can also be finitely represented, resulting in a linear program. This has been extensively studied, see for example \cite{peyre2019computational} for a recent account of computational aspects.
    
The primal problem \eqref{PP} is written as a finite dimensional linear program over matrices \cite[Section 2.3]{peyre2019computational}:
\begin{equation}
    \min\Big\{
        \langle C,P \rangle\:\Big|\:
        P\in \R_+^{n\times m},\:
        \sum_{j=1}^mP_{i,j}=a_i,\:
        \sum_{i=1}^nP_{i,j}=b_j,\:
        \forall i=1,\ldots,n,\:j=1,\ldots,m
    \Big\},
    \label{eq:PP_discr}
    \tag{PP-discr}
\end{equation}
where $C\in\R^{n\times m}$ is a cost matrix defined as $C_{i,j}=\frac{1}{2}\|x_i-y_j\|^2$ for all $i=1,\ldots,n,\:j=1,\ldots,m$. This is obtained by identifying the matrix $P$ above with a coupling $\gamma$ between $\hat\mu$ and $\hat\nu$ as in \Cref{th:classical_duality_theorem}, with the identification $\int f(x,y) d\gamma(x,y) = \sum_{i,j} f(x_i,y_j) P_{i,j}$ for all measurable $f$.

The corresponding dual linear program is given as follows \cite[Section 2.5]{peyre2019computational}
\begin{equation}
    \max\left\{ 
        \langle f,a \rangle+\langle g,b \rangle \Big|\:
        f\in\R^n,\, g \in\R^{m},\:
        f_i+g_j\leq C_{i,j},\:
        \forall i=1,\ldots,n,\:j=1,\ldots,m
    \right\}.
    \label{eq:DP_discr}
    \tag{DP-discr}
\end{equation}
An application of strong duality ensures that the values of \eqref{eq:PP_discr} and \eqref{eq:DP_discr} are the same and both are attained \cite[Proposition 2.4]{peyre2019computational}. Unlike the coupling matrix $P$ in \eqref{eq:PP_discr}, the dual variables in \eqref{eq:DP_discr} do not have a direct interpretation in terms of convex functions as in \Cref{th:classical_duality_theorem}.
\section{Main results on the estimation of Brenier potentials and Brenier maps}

\subsection{An estimator based on discrete dual solutions}
\label{sec:estimator}

We are now in position to describe our Brenier potential estimator, which is based on the following explicit Brenier potential.
\begin{definition}[Explicit Brenier potential]
    Let $\hat\mu,\hat\nu$ be two discrete measures over $\R^d$, such that $\hat\mu=\sum_{i=1}^na_i\delta_{x_i}$, $\hat\nu=\sum_{j=1}^mb_j\delta_{y_j}$, where $a_i> 0$, $i=1 \ldots n$, $b_j>0$, $j=1,\ldots, m$, $\sum_{i=1}^na_i= \sum_{j=1}^mb_j=1$ and $x_1,\ldots,x_n,\:y_1,\ldots,\:y_m$  atoms in $\R^d$. Let $g \in \R^m$ be an optimal solution for \eqref{eq:DP_discr}, the associated Brenier potential is given by
    \begin{equation}\label{eq:brenier_potential}
        \varphi_g \colon x \in\R^d\mapsto \max_{k=1,\ldots,m}\langle x,y_k \rangle+g_k-\frac{1}{2}\|y_k\|^2.
    \end{equation}
    \label{def:mainEstimator}
\end{definition}
Complementary slackness for linear programming allows to verify the optimality condition from \Cref{th:classical_duality_theorem}.
\begin{lemma}[Optimality of the explicit potential]
	Let $\hat\mu,\hat\nu$ be as in \Cref{def:mainEstimator}, let $P$ be an optimal matrix in the primal \eqref{eq:PP_discr} and $(f,g)$ be an optimal pair of dual variables in \eqref{eq:DP_discr}. Then 
	\begin{align*}
		P_{i,j}>0\implies y_j\in \partial\varphi_g(x_i)
	\end{align*}
	where $\varphi_g$ is given in \Cref{def:mainEstimator}. As a consequence $\varphi_g$ is optimal for \eqref{DP} with $\mu = \hat\mu$, $\nu =  \hat\nu$, and
    \begin{align*}
        \min_{\varphi\in \Gamma_0(\R^d)}\, \int \varphi d\hat\mu + \int\varphi^* d\hat\nu = \min_{g \in \R^m}\, \int \varphi_g d\hat\mu + \int\varphi_g^* d\hat\nu.
    \end{align*}
    \label{lem:explicitBrenierPotential}
\end{lemma}

\begin{proof}
    The constraint in \eqref{eq:DP_discr} writes for all $k=1,\ldots,m$, $i = 1,\ldots, n$, $f_i\leq \frac{1}{2}\|x_i-y_k\|^2-g_k$. 
    For any $i=1,\ldots,n,\:j=1,\ldots,m$ such that $P_{i,j}>0$, complementary slackness \cite[Proposition 3.2]{peyre2019computational} implies that $ f_i+g_j=\frac{1}{2}\|x_i-y_j\|^2$, and $f_i=\min_{k=1,\ldots,m}\frac{1}{2}\|x_i-y_k\|^2-g_k$ with minimizing index $j$. This results in the following implications.
    
    \begin{align*}
        \qquad P_{i,j}>0&\implies\min_{k=1,\ldots,m}\frac{1}{2}\|x_i-y_k\|^2-g_k=\frac{1}{2}\|x_i-y_j\|^2-g_j\\
        &\iff \max_{k=1,\ldots,m} \langle x_i,y_k\rangle-\frac{\|y_k\|^2}{2}+g_k=\langle x_i,y_j\rangle-\frac{\|y_j\|^2}{2}+g_j\\
        &\iff \varphi_g(x_i)=\langle x_i,y_j \rangle-\frac{\|y_j\|^2}{2}+g_j\quad\implies\quad y_j\in \partial \varphi_g(x_i),
    \end{align*}
    where we used that, classically (see \textit{e.g.} \cite[Theorem 23.9]{rockafellar1970convex}), $\partial \varphi_g(x_i) = \conv \{y_j,\, \varphi_g(x_i)=\langle x_i,y_j \rangle-\frac{\|y_j\|^2}{2}+g_j\}$ .
    We identify $P$ with an optimal coupling $\gamma\in\Pi(\hat{\mu},\hat{\nu})$ so that $P_{i,j}>0 \iff (x_i,y_j)\in \supp (\gamma)$. This allows us to use the third point of Theorem \ref{th:classical_duality_theorem} showing that $\varphi_g$ is  an optimal Brenier potential for \eqref{DP}. 
\end{proof}
The potential $\varphi_g$ in \Cref{def:mainEstimator} is convex and globally Lipschitz. As such, its gradient is well defined Lebesgue almost everywhere, for example using Rademacher's theorem. 
Setting for all $x \in \R^d$, $I(x) = \{j = 1,\ldots, m| \varphi_g(x) = \langle x,y_j \rangle+g_j-\frac{1}{2}\|y_j\|^2\}$, the (non-empty) set of active indices, we then have $\partial \varphi_g(x) = \conv\{y_j,\, j \in I(x)\}$ for all $x \in \R^d$.
As the maximum of affine functions, $\partial \varphi_g$ is a singleton almost everywhere, and  any selection is a transport map, extended to $\R^d$, for discrete measures. Note that any selection in the subdifferential is measurable. Indeed, by \cite[Theorem 25.5]{rockafellar1970convex}, it coincides  with $\nabla \varphi_g$ on a set of full Lebesgue measure, and, on this set, $\nabla \varphi_g$ is continuous, hence measurable.
Moreover, any selection may not coincide with $\nabla\varphi_g$ only on a Lebesgue-negligible set, hence measurable by completeness. This ensures that our estimator is $\mu$-measurable by absolute continuity of $\mu$ with respect to the Lebesgue measure.
In what follows we will consider the following simple selection.
\begin{definition}[Transport map extension]
    Let $\hat\mu,\hat\nu$ be as in \Cref{def:mainEstimator}, and let $(f,g)$ be optimal in \eqref{eq:DP_discr}.
    Any  selection $\hat T \colon \R^d \to \R^d$, such that $\hat T(x) \in \partial \varphi_g(x)$ for all $x \in \R^d$ is a transport map extension. For example, 
    \begin{equation*}
        \hat T(x) = \frac{\sum_{j \in I(x)} y_j}{\sum_{j \in I(x)} 1}, \qquad x \in \R^d.
    \end{equation*}
\label{def:transportmap_estimator}
\end{definition}

In the following sections, we will consider the transport map defined in \Cref{def:transportmap_estimator} for empirical probability measures $\hat\mu_n,\hat\nu_m$, which defines our proposed estimator.

\begin{remark}
    The estimator proposed in \cite{seguy2018large} and analyzed in \cite{pooladian2021entropic,pooladian2023minimax} is based on the entropic regularized version of \eqref{eq:PP_discr}, a widespread approach in numerical optimal transport 
    \cite{cuturi2013sinkhorn,peyre2019computational}. With entropic regularization, the dual in \eqref{eq:DP_discr} remains very similar, the constraint being replaced by an exponential penalty term in the functional. Using the main result of \cite{cominetti1994asymptotic}, see also \cite{genevay2016stochastic,weed2018explicit}, as the regularization penalty parameter $\epsilon$ vanishes, the primal-dual solutions of the entropic regularized problem converge to primal-dual solutions of the unregularized problem. Combining this with asymptotic properties of exponential weights, the estimator in \Cref{def:transportmap_estimator} is exactly the limit as $\epsilon \to 0$ of the estimator of \cite{seguy2018large, pooladian2021entropic,pooladian2023minimax} based on entropic regularization.

    From a practical perspective, the main advantage of the estimator in \Cref{def:mainEstimator} compared to entropic regularization is that there is no regularization parameter $\epsilon$ to tune, and that it bypasses the hurdles of running a variant of the Sinkhorn algorithm for small $\epsilon$, typically required for transport map estimation, a notorious numerical issue \cite{peyre2019computational,chhaibi2025faster}. On the other hand, having one less hyperparameter to tune may be a disadvantage as a well tuned $\epsilon$ should perform at least as good as the vanishing limit.
    \label{rem:entropicRegularizationLimit}
\end{remark}

\begin{remark}[On the uniqueness of the estimator]

    As is well known in optimal transport, there can be several optimal solutions for the dual vector $g$ in the discrete problem, potentially leading to different Monge map estimators, even when enforcing $g_m=0$ to factor out the invariance up to additive constants.
    
    Uniqueness of dual potentials is of considerable interest from both a computational (e.g., reproducibility) and a statistical perspective, as it plays a role in establishing asymptotic Gaussianity of optimal transport quantities; see, \textit{e.g.}, \cite{del2024central}. However, in the discrete (empirical) case, it is easy to construct examples where the configuration of the supports leads to several optimal dual variables. Even when the target measure is continuous, uniqueness of dual potentials can be difficult to establish; see \cite{staudt2025uniqueness}. 
    
    In practice, when several optimal solutions exist, the particular solution returned depends on the solver and its implementation (e.g., vertices for the simplex algorithm, a centroid of a face for interior-point methods, or a randomly selected solution for stochastic algorithms). Our theory applies to any choice of the optimal dual potentials. We are not aware of generally applicable strategies that would allow one to select a particular optimal vector $g$ in the definition of our Monge map estimator across a wide family of solvers. We conjecture that most discrete configurations only lead to unique dual vectors (up to the usual invariance). We leave a more precise investigation for future work.
\end{remark}

\subsection{Statistical setting and convergence of the Monge map estimator}
\label{sec:mainStatisticalResult}

In this section, we consider two compactly supported population measures, the source having a density. In a statistical setting, they are approximated by empirical measures, obtained from i.i.d samples, as summarized in the following.
\begin{assumption}
    Let $R>0$ and $M>0$ be scalars such that $\mu$, $\nu$ are compactly supported in $B_R$, and $\mu$ has density upper-bounded by $M$. For any positive integers $n$ and $m$, consider i.i.d random samples $X_1\ldots,X_n$  and $Y_1,\ldots,Y_m$ with respective distribution $\mu$ and $\nu$. We then define the empirical probability measures $\hat{\mu}_n=\frac{1}{n}\sum_{i=1}^n\delta_{X_i}$ and $\hat{\nu}_m=\frac{1}{m}\sum_{i=1}^m\delta_{Y_i}$. The Monge map estimator is $\hat{T}_{n,m}$ as in \Cref{def:transportmap_estimator} for empirical discrete measures $\hat{\mu}_n$ and $\hat{\nu}_m$.
    \label{Assumption_1}
\end{assumption}

Our main result relies on a novel error bound, adapted to our non-smooth estimator. Complete proof details  can be found in \Cref{sec:errorBound}.

\begin{theorem}[\Cref{cor:simplerForm2}]
For any $d,R,M$ there is $C$ positive such that the following holds. Let $\mu$, $\nu$ be probability measures compactly supported in $B_R$ such that $\mu(A) \leq M \lambda(A)$ for any measurable set $A \subset \RR^d$. Let $\bar{\varphi}$ be an $R$-Lipschitz Brenier potential for $\mu$ and $\nu$, optimal for \eqref{DP} as in \Cref{th:classical_duality_theorem}.  For any $\hat{\mu}$, $\hat{\nu}$ supported on $B_R$ with $\phi_0 \in \Gamma_0(\R^d)$ an associated  Brenier potential which is $2R$-Lipschitz on $B_{2R}$, we have
    \begin{align*}
       \int_{\R^d} \| \nabla \bar{\varphi}(x) - \nabla \phi_0(x)\| d\mu(x) &\leq C\left(W_1(\mu,\hat{\mu}) + W_1 (\nu, \hat{\nu})\right)^{\frac{1}{4}}. 
    \end{align*}    
    \label{th:simplerForm2}
\end{theorem}

Our main statistical result for the proposed estimator is an application of \Cref{th:simplerForm2}. 
\begin{theorem}
    Under \Cref{Assumption_1} for all $M>0,\:R>0$ and $d\in \mathbb{N}^*$ there exists $C>0$ such that the following holds. Let  $\bar{\varphi}$ be an optimal Brenier potential, optimal in \eqref{DP} for measures $\mu$ and $\nu$. For measures $\hat{\mu}_n$ and $\hat{\nu}_m$, $n$, $m\in \mathbb{N}^*$. Then 
    \begin{align*}
        &\mathbb{E}\left[\int_{\R^d}\|\nabla \bar{\varphi}(x)-\hat{T}_{n,m}(x)\|d\mu(x)\right]\\
        &\leq C
        \begin{cases}
            \left( n^{-1/2}+m^{-1/2}\right)^{1/4} & \text{if}\: d=1 ,\\
            \left(n^{-1/2}\log(1+n)+m^{-1/2}\log(1+m)\right)^{1/4} & \text{if}\: d=2,\\
            \left( n^{-1/d} +m^{-1/d}\right)^{1/4} & \text{if } d > 2.
        \end{cases}
    \end{align*}
    \label{th:main_stat_result}
\end{theorem}

\begin{proof}
    From Definition \ref{def:transportmap_estimator} there exists $g\in \R^m$ and a function $\varphi_g:x\mapsto\max_{i=1,\ldots,m} \allowbreak \langle x,y_i \rangle+g_i-\frac{\|y_i\|^2}{2}$ for which $\hat{T}_{n,m}$ is a selection of its subdifferential.  This function is globally $R$-Lipschitz (in particular $2R$-Lipschitz on $B_{2R}$). We can thus apply \Cref{th:simplerForm2} to this Brenier potential along with Jensen's inequality  and obtain 
    \begin{align*}
        \left[\mathbb{E} \left(\int_{\R^d}\|\nabla \bar{\varphi}(x)-\hat{T}_{n,m}(x)\|d\mu(x)\right)
        \right]^4&\leq \mathbb{E}\left[ \left(\int_{\R^d}\|\nabla \bar{\varphi}(x)-\hat{T}_{n,m}(x)\|d\mu(x)\right)^4\right] \\
        &\leq C^4\mathbb{E}\left[\left(W_1(\mu,\hat{\mu}_n) + W_1 (\nu, \hat{\nu}_m)\right)\right]
    \end{align*}
    where $C$ is the multiplicative constant of Corollary \ref{cor:simplerForm2}.
    We now apply \cite[Main Theorem]{fournier2015rate}. For any integer $q\geq 3$ we have the following inequality 
    \begin{equation*}
        \mathbb{E}\left[W_1(\mu,\hat{\mu}_n)\right] \leq K_{d,q}\left(\int_{\R^d}\|x\|^qd\mu(x)\right)^{1/q}\times
            \begin{cases}
            n^{-1/2}+n^{-1+1/q}& \text{if}\: d=1 ,\\
            n^{-1/2}\log(1+n)+n^{-1+1/q} & \text{if}\: d=2,\\
            n^{-1/d} +n^{-1+1/q} & \text{if } d > 2
        \end{cases}
    \end{equation*}
    where $K_{d,q}>0$ is a constant depending only on $d$ and $q$ given in \cite{fournier2015rate}. We choose $q=3$ for example and upper bound the dominated term $n^{-2/3}$ by $n^{-1/2}$. This yields the result.
    
\end{proof}

\begin{remark}
    We also describe an $L^2$ error bound in \Cref{cor:simplerform2_L2}. Therefore similar statistical rates as in \Cref{th:main_stat_result} also hold in $L^2$, with an exponent $1/5$.
\end{remark}

\begin{remark}
    The previous theorem can easily be adapted to the case where only one measure is empirical. For example, in Section \ref{sec:semi-discr}, we consider the semi-discrete case in which $\mu$ is approximated by samples and $\nu$ is a fixed discrete probability measure.
\end{remark}

\begin{remark}
    Let us consider the case where $\mu$ admits an $L^1$ Poincar\'e inequality, in the sense of \cite[Equation (1.2)]{bakry2008simple}. This holds true for example if $\mu$ has a log-concave density $\exp(-V)$ for a $C^2$, convex lower bounded $V$ \cite[Corollary 1.9]{bakry2008simple} or a bounded density away from $0$ on a convex domain \cite{acosta2004optimal}. In these cases, the explicit empirical Brenier potential in \Cref{def:mainEstimator} converges to a Lipschitz optimal Brenier potential $\varphi$ from \Cref{th:classical_duality_theorem}, up to an additive term, at the same rate as in \Cref{th:main_stat_result}. Indeed, any Lipschitz function can be approximated, together with its gradient, by a smooth function. \Cref{th:main_stat_result} allows to control the $L^1$ distance of gradients, and Poincar\'e inequality gives, up to an additive constant, a control of the $L^1$ distance of potentials. The resulting rates on convergence of potentials are probably suboptimal.
    \label{rem:poincare}
\end{remark}

\subsection{Comparison with existing art on Monge map estimation}
\label{sec:comparisonStat}

For simplicity,  we discuss the case $n=m$ and $d>2$ corresponding to case 3 of \Cref{th:main_stat_result}. Our convergence rate is related to that of the empirical Wasserstein distance. It has been established in \cite{fournier2015rate} that $\mathbb{E}[W_1(\mu,\hat{\mu}_n)]=O\left( n^{-\frac{1}{d}}\right)$, and this rate is attained for the uniform distribution on the hypercube. Our rate is similar, up to the $1/4$ exponent, which arises from the stability result of \Cref{cor:simplerForm2} (see also \Cref{sec:connection_literature_stability}). The literature on convergence rates of transport map estimators is vast, we start with \cite{manole2024plugin} which describes $n^{-1/d}$ rates in $L^2$ norm. They are obtained under the following assumption. 
\begin{assumption}[Monge diffeomorphism assumption]
Let $\mu$ and $\nu$ be two compactly supported measures. There exists an optimal Brenier potential in \eqref{DP} whose gradient is a bi-Lipschitz diffeomorphism.
    \label{assumption:bi_lip_diffeo}
\end{assumption}
\begin{remark}\label{rem:bi_lip_diffeo}
    In \cite{manole2024plugin} this assumption is written as a condition on the Hessian of the potential given by $\frac{1}{\lambda}I_d\preceq \nabla^2\varphi \preceq \lambda I_d $, for some $\lambda>0$.
\end{remark}

The obtained rate in \cite{manole2024plugin} is min-max optimal against the Monge diffeomorphisms class. Several estimators are proposed. One is based on one-nearest neighbor combined with barycentric projection, computationally efficient, but not the gradient of a convex function. The second is a least-squares estimator constrained to the class of gradients of convex functions, with a computational overhead compared to mere computation of the empirical Wasserstein distance, and in addition, for which $\lambda$ needs to be known. 

In this context, our estimator $\hat{T}_{n,m}$ and its analysis in \Cref{th:main_stat_result} has substantial advantages described in the following.
\begin{enumerate}
    \item Since $\hat{T}_{n,m}$ is the sub-gradient of a convex function, it serves as a natural estimator for Monge maps as it is the smallest class of functions containing all such maps.
    \item It does not require any computation beyond dual optimal $g\in\R^m$ in \Cref{def:mainEstimator} from the discrete dual problem. This is given as a byproduct for many solvers.
    \item Our assumptions on the population Brenier potentials are much more applicable compared to \Cref{assumption:bi_lip_diffeo}. In particular we do not even require continuity of the Monge map. We are not aware of any quantitative analysis in this setting.
\end{enumerate}
Yet, our analysis retains statistical rates similar to the Wasserstein rate up to the $1/4$ exponent which is independent of the dimension.  It would be interesting to know if this rate could be further refined to $n^{-1/d}$ for our estimator under \Cref{assumption:bi_lip_diffeo}, this question is left for future work.

Beyond \cite{manole2024plugin}, let us mention other works. The first contribution to consider statistical estimation of Monge maps based on the dual transport problem is \cite{hutter2021minimax}, which also works under \Cref{assumption:bi_lip_diffeo}. The authors establish min-max results for transport map estimators whose convergence rate approaches $n^{-1/2}$ in $L^2$ norm for Monge maps of high differentiability order. These estimators become computationally expensive as the dimension grows. Another example is \cite{pooladian2021entropic} which considers entropic transport maps, and \cite{divol2025optimal} for estimation in abstract function spaces (e.g. shallow neural networks and RKHS), both under \Cref{assumption:bi_lip_diffeo}. We are not aware of any quantitative analysis beyond this assumption, in particular allowing for possibly discontinuous Monge transport maps, except \cite{pooladian2023minimax} which is limited to the semi-discrete case.

Let us now discuss explicit sufficient conditions on the source and target measures ensuring that \Cref{assumption:bi_lip_diffeo}, given in abstract form, holds true. This is actually one of the important challenges in optimal transport beyond statistical estimation, see \cite[Remark 3]{hutter2021minimax}, \cite[Section 2]{manole2024plugin} and \cite[Section 12]{villani_opt_old_new}. The most widely known sufficient conditions on the measures is Caffarelli's celebrated result on smoothness of Monge maps \cite{caffarelli1992boundary,caffarelli1992regularity,caffarelli1996boundary} (summarized in \cite[Theorem 12.50]{villani_opt_old_new}). It states that, if both the source and target measures are supported on compact convex sets with H\"older continuous densities, bounded away from $+ \infty$ and $0$, then the Monge diffeomorphism assumption holds true. The scope of applicability is much narrower compared to ours: compactly supported measures, one of them with an upper bounded density.

\subsection{The semi-discrete case}\label{sec:semi-discr}
In this section, we consider an absolutely continuous source measure $\mu$ and a fixed discrete target measure $\nu$, a setting known as semi-discrete optimal transport, as summarized in the following assumption.
\begin{assumption}
    Let $R>0$ and $M>0$, $\mu$, $\nu$ be supported in $B_R$, such that $\mu$ has density upper-bounded by $M$. Assume further that $\nu=\sum_{j=1}^m b_j\delta_{y_i}$ for some fixed integer $m>1$, support points $y_1,\ldots,y_m\in \R^d$ and weight vector $b_1,\ldots, b_m > 0$ such that $\sum_{j=1}^m b_j = 1$. For any positive integers $n$, consider i.i.d random samples $X_1,\ldots,X_n$ distributed according to $\mu$. We then define the empirical measure $\hat{\mu}_n=\frac{1}{n}\sum_{i=1}^n\delta_{X_i}$. Our Monge map estimator is $\hat{T}_n$ as in \Cref{def:transportmap_estimator} for discrete measures $\hat{\mu}_n$ and $\nu$.
    \label{assumption:semi_discrete}
\end{assumption}
Using similar ideas as in \Cref{lem:explicitBrenierPotential}, the dual formulation in \Cref{th:classical_duality_theorem} in the semi-discrete setting has a finite dimensional equivalent expression. This type of argument is typical of the semi-discrete setting \cite{levy2024large} and is most often presented with Kantorovich potentials. We use here the dual framework of \Cref{th:classical_duality_theorem} and present proof details \red{in \Cref{sec:semidual_semidiscrete}} for completeness.

\begin{lemma}
    Let $\mu$ be a probability measure on $\R^d$, and $\nu$ be as in \Cref{assumption:semi_discrete}, both with compact support in $B_R$. For any $g \in \R^m$, let $\varphi_g \colon x \mapsto \max_{k=1,\ldots,m}\langle x,y_k \rangle+g_k-\frac{1}{2}\|y_k\|^2$ as in \Cref{def:mainEstimator}.
    The following three minimization problems have the same value, each of them being attained:
    \begin{align}
        \min_{\varphi \in \Gamma_0(\R^d)} &\int \varphi d\mu + \int \varphi^* d\nu \label{eq:semiDiscreteSemiDual1}\\ 
        \min_{g \in \R^m} \,\,\,\,
        &\int  \varphi_g d\mu + \int \varphi_g^* d\nu \label{eq:semiDiscreteSemiDual2}\\
        \min_{g \in \R^m} \,\,\,\,&\int  \varphi_g d\mu + \sum_{j=1}^m b_j \left(\frac{1}{2} \|y_j\|^2- g_j\right).\label{eq:semiDiscreteSemiDual3}
    \end{align}
    Furthermore, for $g\in\R^m$ optimal in \eqref{eq:semiDiscreteSemiDual3}, we have $\varphi^*_{g} (y_j) =\frac{1}{2} \|y_j\|^2- g_j$ for all $j = 1,\ldots, m$, and $g_j$ can be chosen in $[-3R^2,3R^2]$, for all $j=1,\ldots,m$, with $g_m = 0$.
    \label{lem:semidiscreteSemidual} 
\end{lemma}

The finite-dimensional parametrization in \Cref{lem:semidiscreteSemidual} allows us to adapt \Cref{th:main_stat_result}, yielding an improved statistical rate, with dimension free-exponent.
\begin{theorem}
	\label{th:mainResultSemiDiscrete}
	For all $M>0$, $R>0$, and $d\in \mathbb{N}^*$ there exists $C>0$ such that for all $\mu$, $\nu$ as in \Cref{assumption:semi_discrete} the following holds. Let $\bar{\varphi}$ be an optimal Brenier potential in \eqref{DP} for measures $\mu$ and $\nu$. Then, for $\nu$ fixed,
    \begin{equation*}
        \mathbb{E}\left[ \int_{\R^d}\|\nabla \bar{\varphi}(x)-\hat{T}_n(x)\|d\mu(x)\right]\leq C\left( \frac{m}{n}\right)^{\frac{1}{8}}.
    \end{equation*}
\end{theorem}
\begin{proof}
    Invoking \Cref{lem:semidiscreteSemidual}, we have $\bar{g}, \hat{g}$ optimal for \eqref{eq:semiDiscreteSemiDual3}, with infinite norm at most $3R^2$, associated respectively to the semi-discrete problem for $(\mu,\nu)$ and the discrete problem for $(\hat{\mu}_n,\nu)$. The resulting potentials, $\varphi_{\bar{g}}$ and $\varphi_{\hat{g}}$ are optimal for \eqref{eq:semiDiscreteSemiDual2} and are therefore optimal dual potentials in \Cref{th:classical_duality_theorem}. Finally we may consider that \eqref{eq:semiDiscreteDualPotential} holds as equality for all $i = 1, \ldots, m$ for both $\varphi_{\bar{g}}$ and $\varphi_{\hat{g}}$. We have

    \begin{align*}      &\int_{\R^d}\varphi_{\hat{g}}d\mu+\int_{\R^d}\varphi_{\hat{g}}^*d\nu-\left(\int_{\R^d} \bar{\varphi}d\mu +\int_{\R^d}\bar{\varphi}^*d\nu\right)\\
        =\;&\int_{\R^d} \varphi_{\hat{g}} d\mu +\sum_{i=1}^m b_i\left( \frac{1}{2}\|y_i\|^2-\hat{g}_i  \right)-\left(\int_{\R^d} \varphi_{\bar{g}} d\mu + \sum_{i=1}^m b_i\left( \frac{1}{2}\|y_i\|^2-\bar{g}_i  \right)\right)\\
        \leq\;& \int_{\R^d} \varphi_{\hat{g}} d\mu +\sum_{i=1}^m b_i\left( \frac{1}{2}\|y_i\|^2-\hat{g}_i  \right)-\left(\int_{\R^d} \varphi_{\hat{g}} d\hat{\mu}_n +\sum_{i=1}^m b_i\left( \frac{1}{2}\|y_i\|^2-\hat{g}_i  \right)\right)\\
        &+\int_{\R^d} \varphi_{\bar{g}} d\hat{\mu}_n +\sum_{i=1}^m b_i\left( \frac{1}{2}\|y_i\|^2-\bar{g}_i  \right)-\left(\int_{\R^d} \varphi_{\bar{g}} d\mu +\sum_{i=1}^m b_i\left( \frac{1}{2}\|y_i\|^2-\bar{g}_i\right)  \right)\\
        \leq\;& 2\sup_{g \in [-3R^2,3R^2]^m }\left| \int_{\R^d} \varphi_{g}d\mu -\int_{\R^d}\varphi_g d\hat{\mu}_n \right|
    \end{align*}
    where we used for the first inequality the fact that $\hat{g}$ is a minimizer in \eqref{eq:semiDiscreteSemiDual3} for measure $\hat{\mu}_n$.   We apply  \Cref{cor:simplerForm} to the potentials $\bar{\varphi}$ and $\varphi_{\hat{g}}$ and Jensen's inequality. Recalling that $\hat{T}_n=\nabla \varphi_{\hat{g}}$ almost everywhere, it yields
    \begin{equation}
        \frac{1}{C^4}\mathbb{E}\left[\int \|\nabla \bar{\varphi}(x)-\hat{T}_n(x) \|d\mu(x) \right] ^4\leq\mathbb{E}\left[2\sup_{g\in [-3R^2,3R^2]^m }\left| \int_{\R^d} \varphi_{g}d\mu -\int_{\R^d}\varphi_g d\hat{\mu}_n\right| \right].
        \label{eq:proof_semi-discrete_stat_thm}
    \end{equation}
    To upper bound the expectation of the last term we are going to use \cite[Theorem 3.5.1]{gine2021mathematical} with the following class of functions
    \begin{equation*}
        \mathcal{F}=\{x\mapsto\varphi_g(x),\:g\in \left[-3R^2,3R^2\right]^m\}.
    \end{equation*}
    Notice that for all $x\in\R^d$ and for any $w,v\in \R^m$ we have $|\varphi_w(x)-\varphi_v(x)|\leq \|w-v\|_\infty$ and thus $\|\varphi_w-\varphi_v\|_{L^2(\mu_n)}  \leq \|w-v\|_\infty$. 
    
    We thus verify the following chain of inequalities for covering numbers
    \begin{align*}
            &\mathcal{N}(\tau/2,\mathcal{F},\Vert\cdot\Vert_{L^2(\mu_n)})\\
            &\leq  \mathcal{N}(\tau/2,[-3R^2,3R^2]^m,\Vert\cdot\Vert_{\infty}) 
            \leq \mathcal{N}(\tau/2,\left[-3R^2,3R^2\right],\vert\cdot\vert)^m \leq\left(\frac{12R^2}{\tau} \right)^m.
    \end{align*}
    Furthermore, we have for all $x\in B_R$ and for all $g\in [-3R^2,3R^2]^m$ that $|\varphi_g(x)|\leq \frac{9}{2}R^2$ and thus $\sup_{g\in[-3R^2,3R^2]^m}\|\varphi_g\|_{L^2(\mu)}\leq \frac{9}{2}R^2$. Combining \eqref{eq:proof_semi-discrete_stat_thm} with \cite[Theorem 3.5.1]{gine2021mathematical} we then obtain the desired result
    \begin{equation*}
    \frac{1}{C^4}\mathbb{E}\left[\int \|\nabla \bar{\varphi}(x)-\hat{T}_n(x) \|d\mu(x) \right] ^4\leq 16\sqrt{2}\sqrt{\frac{m}{n}}\int_{0}^{\frac{9}{2}R^2}\sqrt{\log\left(\frac{24R^2}{\tau}\right)}d\tau.
    \end{equation*}
\end{proof}

\begin{remark}
    The previous theorem can be extended to the case where $\nu$ is also approximated by samples, yielding the same convergence rate. The proof can be derived from a conditioning argument, as in \cite[Appendix F]{pooladian2023minimax}.
\end{remark}
\begin{remark}
    \Cref{rem:poincare}, which discusses the case where $\mu$ satisfies an $L^1$ Poincaré inequality, also applies to the previous theorem.
\end{remark}

\paragraph*{Comparison with existing semi-discrete Monge map estimators} 
A semi-discrete Monge map estimator based on the entropic optimal transport map between the empirical source and target measures is studied in \cite{pooladian2023minimax}. Similarly to \Cref{sec:mainStatisticalResult}, their analysis assumes that the source population measure admits a density that is bounded both from above and below on its convex compact support. Under this assumption, they establish convergence rate in  $L_2(\mu)$ of order $O\left(\left(m/n\right)^{\frac{1}{4}}\right)$   which they show to be minimax optimal in $n$.

By contrast, our analysis results in a convergence rate of order $O\left(\left(m/n\right)^{\frac{1}{8}}\right)$ in $L^1(\mu)$ without requiring convexity of the support nor that the density of the source measure is lower bounded away from zero.
We do not know whether this rate is minimax optimal under these general assumptions. Moreover, it remains unclear whether our convergence rate analysis could be improved under the stronger regularity assumptions considered in \cite{pooladian2023minimax}. Addressing both questions is left for future work. Finally, our estimator compares favorably to the baseline, one nearest neighbor estimator, which suffers from the curse of dimensionality in the semi-discrete setting, as described in \cite{pooladian2023minimax}.

\subsection{Convergence of couplings}
\label{sec:couplings}
Existing works on the convergence of empirical optimal couplings are asymptotic \cite{seguy2018large, segers2022graphical}, and we are not aware of finite sample bounds. The following result provides such a bound, based on a stability estimate for optimal couplings presented in \Cref{sec:stabilityCouplings}. 
\begin{theorem}
    For any $M>0,\:R>0$ and $d\in \mathbb{N}^*$
     under the setting of \Cref{Assumption_1},  there exists $C>0$ such that the following holds. Let $\hat{\gamma}_{n,m}$ be an optimal coupling for measures $\hat{\mu}_n, \hat{\nu}_m$, and $\gamma$ be an optimal coupling for measures $\mu$, $\nu$, then 
         \begin{equation*}
        \mathbb{E}\left[W_2(\gamma,\hat{\gamma}_{n,m})\right]\leq C \begin{cases}
        \left( n^{-1/4}+m^{-1/4}\right)^{1/8} & \text{if}\: d<4 ,\\
        \left(n^{-1/4}\sqrt{\log(1+n)}+m^{-1/4}\sqrt{\log(1+m)}\right)^{1/8} & \text{if}\: d=4,\\
        \left( n^{-1/d} +m^{-1/d}\right)^{1/8} & \text{if } d > 4.
        \end{cases}
    \end{equation*}
    \label{th:stat_result_couplings}
\end{theorem}
\begin{proof}
    We apply \Cref{cor:stabilityCoupling} to coupling $\gamma$ and $\hat{\gamma}_{n,m}$. There exists $C>0$ depending on $d$, $R$ and $M$ such that $W_2(\gamma,\hat{\gamma}_{n,m})^8\leq C(W_2(\mu,\hat{\mu}_n)+W_2(\nu,\hat{\nu}_m))$. As in \Cref{th:main_stat_result}, we apply \cite[Main Theorem]{fournier2015rate} for $W_2$, with $q=4$, yielding the result. Note that an additional $\frac{1}{2}$ factor arises as a byproduct of considering $W_2(\mu,\hat{\mu}_n)$ (resp. $W_2(\nu,\hat{\nu}_m)$) instead of $W_2^2(\mu,\hat{\mu}_n)$ (resp. $W_2^2(\nu,\hat{\nu}_m)$) in \cite{fournier2015rate}.
\end{proof}
\begin{remark}
    This estimate is unlikely to admit significant improvement in the semi-discrete setting. Indeed, as suggested by the stability result for couplings on which \Cref{th:stat_result_couplings} builds, the convergence of the couplings cannot be substantially faster than that of the marginals. Since $W_2(\mu,\hat{\mu}_n)$ is subject to the curse of dimensionality, this is also the case for $W_2(\gamma,\hat{\gamma}_{n,m})$.    
\end{remark}
\section{An error bound for Brenier dual potentials}
\label{sec:errorBound}

\subsection{Preliminary results}
\label{sec:preliminaryErrorBound}
We recall the following \cite[Theorem 2.1]{carlier2025quantitative} on the covering number of singularities of convex functions.
\red{Our first preliminary result is a slight local modification of \cite[Theorem 2.1]{carlier2025quantitative}, stated for convex, globally Lipschitz functions. We provide a short proof for completeness.
\begin{theorem}
Let $\phi \colon \RR^d \to \RR$ be convex and $L,R > 0$, such that $\phi $ is $L$-Lipschitz on $B(0,2R)$. Set $\Sigma_{\eta,\alpha} = \left\{ x \in \RR^d| \mathrm{diam} (\partial \phi(B(x,\eta))) \geq \alpha \right\}$.
	Then for all $\alpha$ and $\eta\in]0,R]$,  we have the following bound on the covering number
	\begin{align*}
		\mathcal{N}( \Sigma_{\eta,\alpha} \cap B(0,R), 8\eta) \leq 48d^2 (R+4\eta)^{d-1} \frac{L}{\alpha \eta^{d-1}}.
	\end{align*}
    \label{th:coveringNumber}
\end{theorem}
\begin{proof}
By \cite{cobzas1978norm}, there exists a convex, globally $L$-Lipschitz function $\Phi$ such that $\Phi|_{B(0,2R)}=\phi|_{B(0,2R)}$, and therefore $\partial \Phi|_{B(0,2R)} = \partial \phi|_{B(0,2R)}$. Furthermore, for  all $\eta\in ]0,R]$ and all $x\in B(0,R)$, we have $B(x,\eta)\subset B(0,2R)$. Therefore
\begin{align*}
     \{x\in B(0,R),\:\mathrm{diam}(\partial \Phi(B(x,\eta)))\geq \alpha\} 
     =\{x\in B(0,R),\:\mathrm{diam}(\partial \phi(B(x,\eta)))\geq \alpha\},
\end{align*}
The conclusion is from \cite[Theorem 2.1]{carlier2025quantitative} applied to the globally Lipschitz function $\Phi$. 
\end{proof}

}

We also recall the definition of the proximal operator on $\R^d$ \cite{moreau1965proximite}.
\begin{definition}
    For any $f\in \Gamma_0(\R^d)$, the proximal operator is defined, for all $v\in \R^d$, as
    \begin{equation*}
        \prox_f(v)=\underset{x\in\R^d}{\argmin} \ f(x)+\frac{1}{2}\|x-v\|^2.
    \end{equation*}
\end{definition}

Finally, we recall the Fenchel-Young inequality with a remainder due to \cite{carlier2023fenchel}. 
\begin{lemma}
	Let $\phi\in \Gamma_0(\R^d)$. For every $x,p\in \R^d$ and $\lambda>0$ we have
	\begin{align*}
		\phi(x) + \phi^*(p) - \left\langle x, p \right\rangle  \geq \frac{1}{\lambda}\|x - \mathrm{prox}_{\lambda \phi}(x + \lambda p)\|^2.
	\end{align*}
	\label{th:Carlier}
\end{lemma}
Extensions around \Cref{th:Carlier} can be found in \cite{bauschke2022carlier,burachik2025note} and in \cite{combettes2026lower} with a historical discussion.

\subsection{A dual error bound}
Leveraging the results presented in \Cref{sec:preliminaryErrorBound}, we obtain the following error bound for convex potentials, relating an integrated Fenchel-Young remainder to $L_1$ distance of gradients. In an optimal transport context, this provides an error bound on the dual problem.

\begin{theorem}
Let $\mu$ be absolutely continuous with density upper bounded by $M$ on $\RR^d$, compactly supported in $B_R$ and $\bar{\varphi},\phi$ be convex lower semi-continuous on $\RR^d$ and $L$-Lipschitz on $B_{2R}$.
	Consider any $q \in [1,2]$ and the following constant:
	\begin{align*}
		C_{R,d} = 4 \frac{e}{\sqrt{2\pi} \sqrt{\frac{d}{2} + 1}}\left(\frac{\pi e 64}{\left( \frac{d}{2} + 1 \right)} \right)^{\frac{d}{2}}  48d^2 (2R)^{d-1}.
	\end{align*}
	\textbf{Case 1:} If
	\begin{align*}
		\int \phi(x) + \phi^*(\nabla \bar{\varphi}(x)) - \left\langle x, \nabla \bar{\varphi}(x) \right\rangle  d\mu(x) > \left(\frac{R}{8}\right)^{\frac{q+3}{q+1}} 2L (MC_{R,d})^{\frac{2}{q+1}}
	\end{align*}
	then
	\begin{align*}
		\| \nabla \bar{\varphi} - \nabla \phi\|_{L^q(\mu)}
		\leq\,&  4 \sqrt{\frac{L}{R}} \left( \int \phi(x) + \phi^*(\nabla \bar{\varphi}(x)) - \left\langle x, \nabla \bar{\varphi}(x) \right\rangle  d\mu(x)  \right)^\frac{1}{2} \\ &+ L\left(\frac{32 }{RL} \int \phi(x) + \phi^*(\nabla \bar{\varphi}(x)) - \left\langle x, \nabla \bar{\varphi}(x) \right\rangle  d\mu(x)\right)^{\frac{q+1}{2q}}.
	\end{align*}
	\textbf{Case 2:} Otherwise
	\begin{align*}
		\| \nabla \bar{\varphi} - \nabla \phi\|_{L^q(\mu)}
		\leq 4\left(L^{q+2}MC_{R,d} \right)^{\frac{1}{q+3}} \left( \int \phi(x) + \phi^*(\nabla \bar{\varphi}(x)) - \left\langle x, \nabla \bar{\varphi}(x) \right\rangle  d\mu(x) \right)^{\frac{1}{q+3}}.
	\end{align*}
	\label{th:dualErrorBound}
\end{theorem}
\begin{remark}
    The following equivalent expressions explicitly establish the connection with the dual problem of optimal transport in \eqref{eq:preliminary_semi_dual}, where $\nu$ is the pushforward of $\mu$ by $\nabla \bar{\varphi}$. This is a consequence of the fact that $\left\langle x, \nabla \bar{\varphi}(x) \right \rangle =  \bar{\varphi}(x) + \bar{\varphi}^*(\nabla \bar{\varphi}(x))$ for $\mu$ almost all $x$, and
    \begin{align*}
        &\int \phi(x) + \phi^*(\nabla \bar{\varphi}(x)) - \left\langle x, \nabla \bar{\varphi}(x) \right\rangle  d\mu(x) \\
        = &\int \phi(x) + \phi^*(\nabla \bar{\varphi}(x))d\mu(x) - \int \bar{\varphi}(x) + \bar{\varphi}^*(\nabla \bar{\varphi}(x)) d\mu(x)\\
        = &\int \phi d\mu + \int \phi^*d\nu - \left(\int \bar{\varphi} d\mu + \int \bar{\varphi}^*d\nu \right).
    \end{align*}
     In this case, the right hand side is the dual suboptimality gap of the optimal transport problem between $\mu$ and $\nu$ in \Cref{th:classical_duality_theorem}.
    \label{rem:connectionDual}
\end{remark}
\begin{proof}
	For any $\lambda > 0$ (to be chosen later), and $\mu$ almost all $x$, set $z_\lambda(x) = \prox_{\lambda \phi}(x + \lambda \nabla \bar{\varphi}(x))$. This is defined by the relation $z_\lambda(x) = (I + \lambda \partial \phi)^{-1}(x + \lambda \nabla \bar \varphi(x))$ or equivalently,
	\begin{align}
		\label{eq:selectionSubdifferential}
		v_\lambda (x):= \nabla \bar{\varphi}(x) + \frac{x- z_\lambda(x)}{\lambda} \in \partial \phi (z_\lambda(x)).
	\end{align}
    We have $\prox_{\lambda \phi}(x + \lambda v) = x$ for any $v \in \partial \phi(x)$. Moreover, the $\prox$ mapping being $1$-Lipschitz, we deduce that for all $x \in B(0,R)$,
	\begin{align}
        \label{eq:definitionEta}
		\|z_\lambda(x) - x \| = \|\prox_{\lambda \phi}(x + \lambda \nabla \bar{\varphi}(x)) - \prox_{\lambda \phi}(x + \lambda v)\|  \leq \lambda \| \nabla \bar\varphi(x) -v \|\leq 2 \lambda L := \eta,
	\end{align}
	where $\eta$ is a constant appearing in the covering number in \Cref{th:coveringNumber} to be chosen later. For the last inequality, we have used the fact that both $\bar{\varphi}$ and $\phi$ are $L$-Lipschitz in a neighborhood of any $x \in \supp\ \mu$. We are going to consider only values $\eta \leq R$, so that $z_\lambda(x) \in B(0,2R)$ for any $x \in \supp\ \mu$, and the corresponding $v_\lambda(x)$ is bounded by $L$ using the Lipschitz continuity of $\phi$ and the characterization \eqref{eq:selectionSubdifferential}.
	Set
	\begin{align*}
		\delta = \int \phi(x) + \phi^*(\nabla \bar{\varphi}(x)) - \left\langle x, \nabla \bar{\varphi}(x) \right\rangle  d\mu(x).
	\end{align*}
	By Carlier's inequality in \Cref{th:Carlier}, we have for any $\lambda > 0$, using \eqref{eq:selectionSubdifferential},
	\begin{align*}
		\delta \geq \frac{1}{\lambda} \int\|x - z_\lambda(x)\|^2 d\mu(x) = \lambda \int\|\nabla \bar{\varphi}(x) - v_\lambda(x)\|^2 d\mu(x).
	\end{align*}
	We deduce, using Jensen inequality with $q \leq 2$, and introducing $\mathbb{I}$ the function that maps a boolean value true/false to the corresponding numerical value 1/0,
	\begin{align}
		\label{eq:ineqTemp1}
		\|\nabla \bar{\varphi} - v_\lambda\|_{L^q(\mu)}^2  \leq \Bigl(\int\|\nabla \bar{\varphi}(x) - v_\lambda(x)\|^2 d\mu(x)\Bigr)^{\frac{q}{2} \times \frac{2}{q}} = \int\|\nabla \bar{\varphi}(x) -  v_\lambda(x)\|^2 d\mu(x) \leq \frac{\delta}{\lambda}.
	\end{align}
    For any $\alpha \geq 0$ (appearing in the covering number in \Cref{th:coveringNumber}, to be chosen later), and $\eta =2L\lambda \leq R$, we have using the definition of $v_\lambda$ in \eqref{eq:selectionSubdifferential},
	\begin{align}
		\int \| v_\lambda - \nabla \phi\|^q d\mu \leq\,& (2L)^q \int \mathbb{I}(\| v_\lambda(x) - \nabla \phi(x)\| \geq \alpha)d\mu(x)\\
        &+ \alpha^q \int \mathbb{I}(\| v_\lambda(x) - \nabla \phi(x)\| < \alpha)d\mu(x)\nonumber\\
		\leq\,& (2L)^q \int \mathbb{I}(\mathrm{diam}(\partial \phi(B(x,\eta))) \geq \alpha) d\mu(x) + \alpha^q\nonumber\\
		\leq\,& (2L)^qM \ \mathrm{Vol}(B(0,8\eta)) \ \mathcal{N}(\Sigma_{\eta,\alpha} \cap B(0,R), 8 \eta) + \alpha^q,
		\label{eq:ineqTemp2}
	\end{align}
	where the first inequality is because both $v_\lambda$ and $\nabla \phi$ are bounded by $L$ on $\supp\ \mu$, the second is because of the characterization \eqref{eq:selectionSubdifferential} and the upper bound \eqref{eq:definitionEta}, and the last one is a covering argument combined with the fact that $\mu$ has a density bounded by $M$, $\Sigma_{\eta,\alpha}$ is given in \Cref{th:coveringNumber}. 
    
	We have, combining \eqref{eq:ineqTemp1} and \eqref{eq:ineqTemp2}, for any $\eta \in (0,R)$ and $\alpha \geq 0$, using $2 L \lambda = \eta$
	\begin{align}
		\| \nabla \bar{\varphi} - \nabla \phi\|_{L^q(\mu)} &\leq \| \nabla \bar{\varphi} - v_\lambda\|_{L^q(\mu)} + \| v_\lambda - \nabla \phi\|_{L^q(\mu)} \nonumber\\
		&\leq \sqrt{\frac{2L\delta}{\eta}} + \left((2L)^q M \ \mathrm{Vol}(B(0,8\eta)) \ \mathcal{N}(\Sigma_{\eta,\alpha} \cap B(0,R), 8 \eta) + \alpha^q \right)^{\frac{1}{q}}\nonumber \\ 
		&\leq \sqrt{\frac{2L\delta}{\eta}} + 2L \left(M \ \mathrm{Vol}(B(0,8\eta)) \ \mathcal{N}(\Sigma_{\eta,\alpha} \cap B(0,R), 8 \eta) \right)^{\frac{1}{q}}+ \alpha , 
		\label{eq:mainIneq}
	\end{align}
	where we have used the subaditivity of $t \mapsto t^{1/q}$ for $q\geq 1$.

	We are going to work the two cases separately. Let us consider the first case $\delta > \left(\frac{R}{8}\right)^{\frac{q+3}{q+1}} 2L (MC_{R,d})^{\frac{2}{q+1}} \geq \frac{R}{32} L (MC_{R,d}R)^{\frac{2}{q+1}}  $. We choose $\eta = \frac{R}{8}$ and $\alpha = 0$. In this case the covering number in \eqref{eq:mainIneq} is $\mathcal{N}(\Sigma_{\eta,\alpha} \cap B(0,R), 8 \eta) = 1$. We expand the right hand side of \eqref{eq:mainIneq}, using the volume estimates of \Cref{Lem:volume_estimate_ball} \red{from the appendix}
	\begin{align*}
		4 \sqrt{\frac{L\delta}{R}} + L \left( 2^qM \ \mathrm{Vol}(B(0,R))\right)^{\frac{1}{q}} &\leq 4 \sqrt{\frac{L\delta}{R}} + L \left(M  \frac{2^qe}{\sqrt{2\pi} \sqrt{\frac{d}{2} + 1}}
	\left(\frac{\pi e}{\left( \frac{d}{2} + 1 \right)} \right)^{\frac{d}{2}} R^d\right)^{\frac{1}{q}}\\
		&\leq 4 \sqrt{\frac{L\delta}{R}} + L(M C_{R,d}R)^{\frac{1}{q}} \leq 4 \sqrt{\frac{L\delta}{R}} + L\left(\frac{32 \delta}{RL} \right)^{\frac{q+1}{2q}}.
	\end{align*}
	This concludes the first point.

	Let us now consider the second case: $\delta \leq \left(\frac{R}{8}\right)^{\frac{q+3}{q+1}} 2L (MC_{R,d})^{\frac{2}{q+1}}$. Using \Cref{th:coveringNumber} and the volume estimate in \Cref{Lem:volume_estimate_ball} \red{from the appendix},  \eqref{eq:mainIneq}  becomes, for any $\eta \in ]0,R[$ and $\alpha \geq 0$,
	\begin{align*}
		\| \nabla \bar{\varphi} - \nabla \phi\|_{L^q(\mu)} 
		\leq\,& L \left(M
		\frac{2^qe}{\sqrt{2\pi}\sqrt{\frac{d}{2} + 1} }
		\left(\frac{\pi e 64}{\left( \frac{d}{2} + 1 \right)} \right)^{\frac{d}{2}}  48d^2 (R+4\eta)^{d-1} \frac{\eta L}{\alpha}  \right)^{\frac{1}{q}}\\
        &+ \alpha + \sqrt{\frac{2L \delta}{\eta}}.
	\end{align*}
	For any $\eta \leq \frac{R}{4}$, we get $R + 4 \eta  \leq 2R$ so that, for any $\alpha \geq 0$,
	\begin{align*}
		\| \nabla \bar{\varphi} - \nabla \phi\|_{L^q(\mu)} &\leq  L\left(LM C_{R,d}\frac{\eta}{\alpha}\right)^{\frac{1}{q}}  + \alpha  + \sqrt{\frac{2L \delta}{\eta}}
	\end{align*}
	Consider $\alpha = L \left( M C_{R,d} \eta \right)^{\frac{1}{q+1}} $, we have
	\begin{align*}
		\| \nabla \bar{\varphi} - \nabla \phi\|_{L^q(\mu)}  &\leq  L\left((M C_{R,d}\eta)^{1 - \frac{1}{q+1}}\right)^{\frac{1}{q}}  + L \left( M C_{R,d} \eta \right)^{\frac{1}{q+1}}  + \sqrt{\frac{2L \delta}{\eta}}\\
        &=  2 L( M C_{R,d}\eta)^{\frac{1}{q+1}} + \sqrt{\frac{2L\delta}{\eta}}
	\end{align*}
	Choosing $\eta = \left(\frac{\sqrt{\delta}}{\sqrt{2L} (M C_{R,d})^{\frac{1}{q+1}}}\right)^{\frac{2q + 2}{q+3}} \leq \frac{R}{8}$, we compute (details in \Cref{lem:technicalComputation})
	\begin{align*}
		\| \nabla \bar{\varphi} - \nabla \phi\|_{L^q(\mu)}  &\leq 2 (2L)^{\frac{q+2}{q+3}} (MC_{R,d} \delta)^{\frac{1}{q+3}} \leq  4\left(L^{q+2}MC_{R,d} \delta\right)^{\frac{1}{q+3}}. 
	\end{align*}
	This concludes the proof of the second point.
\end{proof}

\begin{lemma}
    Assume $\delta,L,M,C > 0$ and $q \in [1,2]$.
	For $\eta = \left(\frac{\sqrt{\delta}}{\sqrt{2L} (M C)^{\frac{1}{q+1}}}\right)^{\frac{2q + 2}{q+3}}$, we have
	\begin{align*}
		2 L( M C\eta)^{\frac{1}{q+1}} + \sqrt{\frac{2L\delta}{\eta}} = 2 (2L)^{\frac{q+2}{q+3}} (MC\delta)^{\frac{1}{q+3}}.
	\end{align*}
	\label{lem:technicalComputation}
\end{lemma}
The proof of this Lemma can be found in \Cref{sec:LemmaPowers}.

\subsection{Simplification and application to stability of optimal transport maps}
We now only consider the case $q=1$. The following holds because under the appropriate assumptions, all quantities of interest are bounded so that the $\frac{1}{q+3} = \frac{1}{4}$ exponent dominates over the exponent $\frac{q+1}{2q} = 1$ in \Cref{th:dualErrorBound}, see also \Cref{rem:connectionDual}.
\begin{corollary}
For any $d,R,M$, there is $C>0$ such that the following holds. Let $\mu$, $\nu$ be probability measures compactly supported in $B_R$ such that $\mu(A) \leq M \lambda(A)$ for any measurable set $A \subset \RR^d$. Let $\bar{\varphi}$ be an $R$-Lipschitz Brenier potential for $\mu$ and $\nu$, optimal for \eqref{DP} as in \Cref{th:classical_duality_theorem}.  Then  for any $\phi \in \mathcal{C}_{2R,2R}$ (as in \Cref{lem:brenier_potentials}), we have
    \begin{align*}
       \int \| \nabla \bar{\varphi}(x) - \nabla \phi(x)\| d\mu(x) &\leq C\left(\int \phi d\mu + \phi^*d\nu - \int \bar{\varphi} d \mu + \bar{\varphi}^* d\nu\right)^{\frac{1}{4}},
    \end{align*}
    where $C= \max\left(268\sqrt{R}, 4(8R^3MC_{R,d})^{1/4}\right)$ with $C_{R,d}$ defined in \Cref{th:dualErrorBound}. 
    \label{cor:simplerForm}
\end{corollary}
\begin{proof}
    We aim to establish an upper bound on $\delta:=\int\phi d\mu+\int \phi^*d\nu-\left( \int \bar{\varphi}d\mu+\int \bar{\varphi}^*d\nu \right)$ using \Cref{th:dualErrorBound}, see \Cref{rem:connectionDual}.
    Let  $v\in\partial\phi(0)$, we have $\|v\|\leq 2R$ since $\phi$ is $2R$-Lipschitz on $B_{2R}$. Since $\phi^*$ is also $2R$ Lipschitz on $B_{2R}$, we have for all $y\in \supp \nu \subset B_R$:
    \begin{align*}
            \phi^*(y)&=\phi^*(y)-\phi^*(v)+\phi^*(v) \leq2R\|y-v\|+\phi^*(v) \leq 6R^2+\phi^*(v).
    \end{align*}
    We also have $0\in\partial\phi^*(v)$, so that $v \in \argmin \phi^*$, and by the Fenchel-Young identity we obtain $\phi^*(v)=-\phi(0)$ and thus for all $x\in \supp \mu \subset B_R$ and all $y\in \supp \nu$:
    \begin{align*}
            \phi(x)+\phi^*(y)&\leq\phi(x)+ 6R^2+\phi^*(v)\\
            &=6R^2+\phi(x)-\phi(0)
            \leq 6R^2+2R\|x\|
            \leq 8R^2.
    \end{align*}
    By \Cref{lem:brenier_potentials}, we may assume that $\bar{\varphi} \in \mathcal{C}_{2R,2R}$. We may repeat the same argument for $\bar{\varphi}$ to obtain $\delta\leq 16R^2:=\bar{\delta}$. Upon applying Theorem \ref{th:dualErrorBound} for $q=1$, we have in case 1, with $L=2R$:
    \begin{align*}
            \int_{\R^d}\|\nabla \phi(x)-\nabla\bar{\varphi}(x)\|d\mu(x)&\leq4\sqrt{\frac{L}{R}}\delta^{\frac{1}{2}}+\frac{32}{R}\delta
            \leq \Bigl(8\sqrt{2}\bar{\delta}^{\frac{1}{4}}+\frac{32}{R}\bar{\delta}^{\frac{3}{4}}\Bigr)\delta^{1/4}
            \leq 268\sqrt{R}\delta^{1/4}
    \end{align*}
 
    And in case 2, we have $\int_{\R^d}\|\nabla \phi(x)-\nabla\bar{\varphi}(x)\|d\mu(x)\leq 4(8R^3MC_{R,d})^{1/4}\delta^{1/4}$, which concludes.
\end{proof}
\Cref{cor:simplerForm} allows to obtain a stability bound for gradients of Brenier potentials.
\begin{corollary}
For any $d,R,M$ there is $C$ positive such that for any $\mu,\nu,\bar{\varphi}$ as in \Cref{cor:simplerForm}, for any $\hat{\mu}$, $\hat{\nu}$ supported on $B_R$ with $\phi_0 \in \Gamma_0(\R^d)$ an associated  Brenier potential which is $2R$-Lipschitz on $B_{2R}$, we have
    \begin{align*}
       \int \| \nabla \bar{\varphi}(x) - \nabla \phi_0(x)\| d\mu(x) &\leq C\left(W_1(\mu,\hat{\mu}) + W_1 (\nu, \hat{\nu})\right)^{\frac{1}{4}}. 
    \end{align*}    
    \label{cor:simplerForm2}
\end{corollary}
\begin{proof}
    From \Cref{lem:identification_lipschitz} \red{in the appendix}, repeating the argument made in the proof of \Cref{lem:brenier_potentials}, there is a Brenier potential $\phi \in \mathcal{C}_{2R,2R}$, such that $\phi = \phi_0$ on $B_{2R}$ and $\phi^* = \phi_0^*$ on $\supp \hat{\nu}$. We deduce that $\nabla \phi = \nabla \phi_0$ on $\supp \mu$.
    Let $C>0$ be as in  \Cref{cor:simplerForm}, we have
    \begin{align*}
        &\frac{1}{C^4} \left(\int \| \nabla \bar{\varphi}(x) - \nabla \phi_0(x)\| d\mu(x)\right)^4  = \frac{1}{C^4} \left(\int \| \nabla \bar{\varphi}(x) - \nabla \phi(x)\| d\mu(x)\right)^4  \\
        \leq\ &\int \phi d\mu + \phi^*d\nu - \int \bar{\varphi} d \mu - \bar{\varphi}^* d\nu\\
        \leq\ & \int \phi d\mu + \phi^*d\nu - \int \bar{\varphi} d \mu -\bar{\varphi}^* d\nu +  \int \bar{\varphi} d \hat\mu + \bar{\varphi}^* d\hat\nu - \int \phi d\hat\mu - \phi^*d\hat\nu \\
        =&\int \phi - \bar{\varphi} d(\mu - \hat{\mu}) + \int \phi^* - \bar{\varphi}^* d(\nu - \hat\nu) \ \leq\  3R(W_1(\mu,\hat{\mu}) + W_1 (\nu, \hat{\nu})),
    \end{align*}
    where the first inequality is by \Cref{cor:simplerForm}, the second is because $\phi$ is optimal for the dual problem associated to $\hat{\mu}$ and $\hat{\nu}$, and the last one is because $\phi - \bar\varphi$ and $\phi^* - \bar\varphi^*$ are both $3R$ Lipschitz so that \cite[Theorem 1.14]{ref_villani_topics} can be applied. 
\end{proof}

\subsection{Connection with existing literature on stability of optimal transport map}\label{sec:connection_literature_stability}
\Cref{cor:simplerForm2} relates to the stability of gradient of Brenier potentials. This is most often stated in terms of $L^2(\mu)$ providing an Hilbertian embedding of the Wasserstein space. We present the $L^2$ version of \Cref{cor:simplerForm2} based on \Cref{th:dualErrorBound}.

\begin{corollary}
    In the same setting as \Cref{cor:simplerForm2}, there exists a constant $C>0$ such that 
    \begin{align*}
       \|\nabla \bar{\varphi}-\nabla \phi_0\|_{L^2(\mu)} &\leq C\left(W_1(\mu,\hat{\mu}) + W_1 (\nu, \hat{\nu})\right)^{\frac{1}{5}}.
    \end{align*}
    \label{cor:simplerform2_L2}
\end{corollary}

\begin{sproof}
    The proof follows the same reasoning as \Cref{cor:simplerForm} and \Cref{cor:simplerForm2}. The term $\delta:=\int\phi d\mu+\int \phi^*d\nu-\left( \int \bar{\varphi}d\mu+\int \bar{\varphi}^*d\nu \right)$ can be bounded by $16R^2$ similarly to \Cref{cor:simplerForm2}. \Cref{th:dualErrorBound} with $q=2$ leads to $\delta$ terms with exponents between $1/5$ and $1$. Boundedness of $\delta$ and the same arguments as in \Cref{cor:simplerForm2} show that, up to constants the smallest exponent, $1/5$ dominates.
\end{sproof}
This result generalizes the ideas of Monge embedding, extensively studied for a fixed source measure \cite{gigli2011holder,merigot2020quantitative,delalande2023quantitative,carlier2025quantitative}. Indeed  \Cref{th:dualErrorBound} implies that
\begin{align}
    \|T_{\nu_1} - T_{\nu_2}\|_{L^2(\mu)} \leq CW_1(\nu_1,\nu_2)^{1/5}
    \label{eq:mongeEmbeddingHolder}
\end{align}
where $T_{\nu_1}, T_{\nu_2}$ are Monge maps between $\mu$ and $\nu_1$, and $\mu$ and $\nu_2$ respectively. In this Monge embedding context, \Cref{cor:simplerform2_L2} has several interesting properties: it allows both the source and target to move, the H\"older exponent is $1/5$ and it does not require convexity of the support or lower boundedness of the density (it does not require a Poincarré-Wirtinger inequality).

These can be discussed in light of existing literature. First, we are not aware of such stability results allowing both the source and target to move. Second \cite[Theorem 4.2]{delalande2023quantitative} provides a $1/6$ exponent for $\mu$ supported on a convex compact set with density bounded away from zero and infinity. These conditions were relaxed in \cite{letrouit2024gluing}. Let us emphasize that the cited results are deeper than the Monge embedding discussed in this section: they also provide stability for the Brenier potentials (not only their gradients), which requires a form of connectedness of the domain with nonvanishing density (hence the Poincarré inequality requirement).

Additionally, the exponent in \Cref{eq:mongeEmbeddingHolder} cannot exceed $1/2$. This was shown in \cite{gigli2011holder}, a different example being provided in \cite[Proposition 2.1]{merigot2020quantitative}. The $1/2$ exponent is achieved under \Cref{assumption:bi_lip_diffeo}, see \cite[Proposition 10]{hutter2021minimax}. Under broader assumptions, the exponent cannot exceed $1/4$ \cite{letrouit2025unstable} which turns out to be optimal as shown by \cite{merigot2026sharp}\footnote{Both works were conducted independently. See also the discussion in introduction and conclusion.}. Finally, let us note that the constant and exponent in \Cref{cor:simplerForm2} or \Cref{cor:simplerform2_L2} directly appear in our main result \Cref{th:main_stat_result}.

\section{A stability bound for optimal couplings}
\label{sec:stabilityCouplings}
This section presents a stability result for couplings. We start with two technical lemmas and then present the main result of the section. \red{The proofs of both technical lemmas follow similar ideas to those in \Cref{sec:errorBound} and are deferred to \Cref{sec:proof_couplings}.} 

\subsection{Preliminary results on perturbation and stability of optimal couplings}
The following lemma allows to perturb the support of optimal couplings with a control in Wasserstein distance.
\begin{lemma}
    Let $\tilde{\mu}$ be compactly supported in $B_R$ and $\tilde{\phi},\bar{\varphi}$ be $L$ Lipschitz on $B_R$. Assume that $\tilde{\nu}$ is another measure, such that $\tilde{\gamma}$ is the optimal coupling between $\tilde{\mu}$ and $\tilde{\nu}$ supported on the graph of $\partial\tilde{\phi}$. Then there is a coupling $\gamma_0$, supported on the graph of $\partial \bar{\varphi}$ with first marginal $\mu_0$ supported in $B_{2R}$, such that 
    \begin{align*}
        W_2^2(\tilde{\gamma}, \gamma_0) &\leq \frac{R^2 + 4L^2}{2LR} \left(\int \bar{\varphi} - \tilde{\phi} d\tilde{\mu} + \int \bar{\varphi}^* - \tilde{\phi}^* d\tilde{\nu}\right)\\
        W_2^2(\tilde{\mu}, \mu_0) &\leq \frac{R}{2L} \left(\int \bar{\varphi} - \tilde{\phi} d\tilde{\mu} + \int \bar{\varphi}^* - \tilde{\phi}^* d\tilde{\nu}\right)
    \end{align*}
    \label{lem:perturbCoulingSupport}
\end{lemma}

The next lemma allows to control the stability of optimal couplings supported on the graph of the same subdifferential, as a function of the variations of the source marginal.
\begin{lemma}
    
	Fix $R > 0$, $M>0$.
	Let $\mu,\beta$ be supported in $B_R$ such that $\mu(A) \leq M\lambda(A)$ for any measurable set $A$. Assume that $\bar{\varphi}$ is convex lower semicontinuous, $L$-Lipschitz on $B_{2R}$. Set $\gamma = (id, \nabla \bar{\varphi})_\sharp \mu$, and assume that $\theta$ is a coupling with first marginal $\beta$ and is supported on the graph of $\partial \bar{\varphi}$. Consider the constant $C_{R,d}$ in \Cref{th:dualErrorBound}.\\
	\textbf{Case 1:} if $\left(\frac{6W_2^2(\mu,\beta)}{(2MC_{R,d})^{2/3}}\right)^{\frac{3}{8}} > \frac{R}{8}$, then 
	\begin{align*}
		W_2^2(\gamma, \theta) &\leq \left( 1 + \frac{256L^2}{R^2}\right)W_2^2(\mu,\beta) + L^2M C_{R,d}\left(\frac{6W_2^2(\mu,\beta)}{(2C_{R,d})^{2/3}}\right)^{\frac{3}{8}}.
	\end{align*}
	\textbf{Otherwise:} 
	\begin{align*}
		W_2^2(\gamma, \theta) 
		&\leq W_2^2(\mu,\beta) +  {\color{black}5 L^2 (2MC_{r,d})^{\frac{1}{2}} } W_2^2(\mu,\beta)^{\frac{1}{4}}.
	\end{align*}
	\label{lem:extensionStability}
\end{lemma}

\subsection{Stability of optimal couplings under variations of source and target}

The following is the main stability result of this section. A simpler corollary is given right after. The main idea of the proof is illustrated in \Cref{fig:illustrProofStability}.
\begin{theorem}
    For any $d,R_\mu,R_\nu,M$, there is $C>0$ such that the following holds.
    Let $\mu$ and $\nu$ be probability measures supported on $B_{R_\mu}$ and $B_{R_\nu}$ respectively, with $\mu(A) \leq M \lambda(A)$ for any measurable set $A\subset\R^d$. Let  $\bar{\varphi} \in \mathcal{C}_{4R_\mu,4R_\nu}$  be an optimal Brenier potential (as in \Cref{lem:brenier_potentials}) for these measures and  $\gamma$ be their optimal coupling. Let $\tilde{\mu}$, $\tilde{\nu}$ be probability measures also supported on $B_{R_\mu}$ and $B_{R_\nu}$ respectively with associated Brenier potential $\tilde{\phi} \in \mathcal{C}_{4R_\mu,4R_\nu}$ and coupling $\tilde{\gamma}$. Then
\begin{align*}
    \frac{1}{C^4} W_2(\gamma,\tilde{\gamma})^4 \leq W_2(\mu,\tilde{\mu}) + \left(\int \bar{\varphi} - \tilde{\phi} d\tilde{\mu} + \int \bar{\varphi}^* - \tilde{\phi}^* d\tilde{\nu} \right)^{\frac{1}{2}}
\end{align*}
    \label{th:stabilityCoupling}
\end{theorem}

\begin{proof}
    We start with \Cref{lem:perturbCoulingSupport} which applies under the hypotheses of  \Cref{th:stabilityCoupling} with the same notation with $R = R_\mu$ and $L = 2 R_\nu$.
    We get $\gamma_0$ supported on $\partial \bar{\varphi}$ with first marginal $\mu_0$ supported on $B_{2R_\mu}$. Since both $\gamma_0$ and $\gamma$ are supported on $\partial \bar{\varphi}$ with marginals in $B_{2R_\mu}$ on which $\bar{\varphi}$ is Lipschitz, \Cref{lem:extensionStability} applies with $\theta = \gamma_0$, $R = 2 R_\mu$ and $L = 4 R_\nu$.

    All Wasserstein distances are bounded by constants depending on $d,M,R_\mu,R_\nu$. Similarly as in the proof of \Cref{cor:simplerForm}, it is possible to obtain $C>0$ by adjusting the constants in \Cref{lem:extensionStability} so that only the smallest power dominates in all cases, resulting in $W_2(\gamma_0, \gamma) \leq CW_2( \mu_0,\mu)^{\frac{1}{4}}$. Note again that both $\gamma_0$ and $\gamma$ are supported on $\partial \bar{\varphi}$ and their first marginal, $\mu_0$ and $\mu$ are supported on $B_{2R_\mu}$ on which $\bar{\varphi}$ is assumed to be Lipschitz. We have
    \begin{align*}
    	W_2(\tilde{\gamma},\gamma) &\leq   W_2(\tilde{\gamma},\gamma_0) + W_2(\gamma_0, \gamma)\\
    	&\leq W_2(\tilde{\gamma},\gamma_0) + C W_2( \mu_0,\mu)^{\frac{1}{4}} \leq W_2(\tilde{\gamma},\gamma_0) + C (W_2( \mu_0,\tilde{\mu})+W_2( \mu,\tilde{\mu}))^{\frac{1}{4}}.
    \end{align*}
    Set $\delta = \int \bar{\varphi} - \tilde{\phi} d\tilde{\mu} + \int \bar{\varphi}^* - \tilde{\phi}^* d\tilde{\nu}$.
    Considering \Cref{lem:perturbCoulingSupport}, we may adjust the constant $C$ such that
    \begin{align*}
    	W_2(\tilde{\gamma},\gamma)&\leq C \sqrt{\delta} + C \left(W_2(\mu,\tilde{\mu}) + \sqrt{\delta}\right)^{\frac{1}{4}}
    \end{align*}
     Similarly to the proof of \Cref{cor:simplerForm}, $\delta$ is bounded by a constant depending only on $d,M,R_\mu,R_\nu$. We may thus keep the smallest power and adjust the constant $C$ to get the desired estimate.
\end{proof} 
\red{From \Cref{th:stabilityCoupling} we obtain the following corollary, whose proof is deferred to \Cref{sec:proof_couplings}.}

\begin{corollary}
    For any $d,R_\mu,R_\nu,M$, there is $C>0$ such that 
for any probability distribution $\mu$, $\nu$ supported on $B_{R_\mu}$ and $B_{R_\nu}$ respectively, with $\mu(A) \leq M\lambda (A)$ and optimal coupling $\gamma$, for any probability distribution $\tilde{\mu}$, $\tilde{\nu}$ also supported on $B_{R_\mu}$ and $B_{R_\nu}$ respectively, with optimal coupling $\tilde{\gamma}$, we have
    \begin{align*}
            \frac{1}{C^8} W_2(\gamma,\tilde{\gamma})^8 &\leq W_2(\mu,\tilde{\mu}) + W_2(\nu,\tilde{\nu}).
    \end{align*}
    \label{cor:stabilityCoupling}
\end{corollary}

\begin{remark}
    The formula above implies uniqueness of the optimal coupling between $\mu$ and $\nu$ (take $\tilde{\mu} = \mu$, $\tilde{\nu} = \nu$). This coincides with the absolute continuity assumption, which is sufficient for uniqueness. 
\end{remark}

\section{Conclusion}
We have proposed an estimator for Monge maps, it is given by the gradient of a convex function and enjoys a closed form expression based on discrete dual solutions. Moreover, we derived finite sample guaranties leveraging a dual error bound for Brenier dual potentials, for both the general and semi-discrete cases. Up to powers, these bounds match existing results in the literature under considerably stronger assumptions. Our results hold in a very general and explicit setting: compact support and upper bounded source density. This allows to consider, in particular, possibly discontinuous Brenier potentials as well as source measures with disconnected support. 

Two questions remain. First, regarding the optimality of our finite sample bounds: could they be improved under our general assumptions or are there matching lower bounds?  Second, regarding the adaptability of our estimator: under the Monge diffeomorphism assumption described in \Cref{sec:comparisonStat}, does it achieve existing minimax rates? These questions are left to future work.

\textbf{On exponent improvements:} After the diffusion of the first version of this work, we became aware of \cite{merigot2026sharp} which also combines results from  \cite{carlier2025quantitative} and convex analysis results in \cite{carlier2023fenchel} to obtain various stability bounds in optimal transport. In particular, the bound in \Cref{cor:simplerform2_L2} is described with an exponent $1/4$, which is optimal \cite{letrouit2025unstable}, instead of $1/5$ in our work. Technically, the difference in exponents comes from the replacement of pointwise estimate in \Cref{th:coveringNumber} with a form integrated over $\alpha$ \cite[Corollary 2.2]{carlier2025quantitative}. We neglected this aspect in the first version of this work and a slight technical modification of our arguments yields improved statistical estimation rates. We mention these rates here, but otherwise leave the main text unchanged in its original form in order to preserve the independence of our contribution and that of \cite{merigot2026sharp}. More details are given in \Cref{sec:improvedRates}.
\begin{itemize}
	\item \Cref{th:main_stat_result}  holds with exponent arbitrarily close to $1/3$ instead of $1/4$. 
    \item  \Cref{th:mainResultSemiDiscrete} holds with exponent arbitrarily close to $1/6$ instead of $1/8$.
    
	\item \Cref{th:stat_result_couplings} holds with exponent $1/6$ instead of $1/8$.
\end{itemize}

\section*{Acknowledgments}
This work benefited from financial support from the French government managed by the National Agency for Research under the France 2030 program, with the references ``ANR-23-PEIA-0004'' and ``ANR-25-PEIA-0002''. EP thanks TSE-P and acknowledges the support of the AI Interdisciplinary Institute ANITI funding, through the ANR under the
France 2030 program (grant ANR-23-IACL-0002), Chair TRIAL, ANR MAD and ANR Regulia. The authors thank the two anonymous  reviewers whose feedback greatly improved our manuscript.

\appendix

\section{Technical Lemmas}
\begin{lemma}
    For all $\varphi_0\in \Gamma_0(\R^d)$ which is $R_\nu$-Lipschitz on $B_{R_\mu}$ there exists $\varphi_1\in \mathcal{C}_{R_\mu,R_\nu}$ such that $\varphi_1=\varphi_0$ over $B_{R_\mu}$, and  $\varphi^*_0 = \varphi^*_1$ on $\partial \varphi_0 (B_{R_\mu})$.
    \label{lem:identification_lipschitz}
\end{lemma}
\begin{proof} 
    Define $\varphi_1:=\varphi_0+\delta(.|\bar{B}_{R_\mu})$ where $\delta(.|C)$ denotes the  indicator function which values $0$ in $C$ and $+\infty$ elsewhere. By construction, we have that $\varphi_1\in \Gamma_0(\R^d)$ and is $R_\nu$-Lipschitz on $B_{R_\mu}$.  
    It also results classically that $\varphi_1^*$ is globally $R_\mu$-Lipschitz. 
    Indeed, $\varphi_1^*$ is finite everywhere as it is defined by maximization of an upper semicontinuous function on a compact set. Therefore, $\partial \varphi_1^*$ has full domain and $x\in \partial\varphi_1^*(y)$ implies $y \in \partial \varphi_1(x)$ (\cite[Theorem 23.5]{rockafellar1970convex}) implying that $\varphi_1(x)$ is finite . By construction of $\varphi_1$ this means that $\|x\|\leq R_{\mu}$ and by \cite[Theorem 24.7]{rockafellar1970convex}, $\varphi_1^*$ is $R_\mu$-Lipschitz.

    For the last part, since $\varphi_0 = \varphi_1$ on $B_{R_\mu}$, it holds that $ \partial \varphi_0 = \partial \varphi_1$ on $B_{R_\mu}$. For any $y \in \partial \varphi_0(B_{R_\mu})$, there is $x \in B_{R_\mu}$ such that $y \in \partial \varphi_0(x)$. Since we have $\varphi_1(x) = \varphi_0(x)$ and also $y \in \partial \varphi_1(x)$, it holds that $\varphi_1^*(y) = \left\langle x, y\right\rangle - \varphi_1(x) = \left\langle x, y\right\rangle - \varphi_0(x) = \varphi_0^*(y)$ using the equality case of Fenchel-Young's inequality \cite[Theorem 23.5]{rockafellar1970convex}.
\end{proof}

\begin{lemma}
    For all  $r>0$ we have the following upper bounds on the volume of balls of radius $r$:
    \begin{align*}
        \mathrm{Vol}(B(0,r)) &\leq \frac{e}{\sqrt{2\pi}\sqrt{\frac{d}{2}+1} }
        \left(\frac{\pi e}{\left( \frac{d}{2} + 1 \right)} \right)^{\frac{d}{2}} r^d \\
        \mathrm{Vol}(B(0,8r)) &\leq
       \frac{e}{\sqrt{2\pi}\sqrt{\frac{d}{2}+1} }
        \left(\frac{\pi e 64}{\left( \frac{d}{2} + 1 \right)} \right)^{\frac{d}{2}} r^d.
    \end{align*}
    \label{Lem:volume_estimate_ball}
\end{lemma}
\begin{proof}
    The Gamma function admits the following lower bound, given for example in \cite[Section 5.6]{olver2010nist}: for all $z>0$, $\Gamma(z) \geq \sqrt{2\pi}z^{z - 1/2}e^{-z}$. Combining this bound with the expression of the volume of $B_{r}$: $\mathrm{Vol}(B_r) = \frac{\pi^{d/2}}{\Gamma\left( \frac{d}{2} + 1 \right)} r^d$ we obtain: 
    \begin{align*}
        \mathrm{Vol}(B_r) &\leq
    	\frac{\pi^{d/2}}{\sqrt{2\pi}\left( \frac{d}{2} + 1 \right)^{\frac{d+1}{2}}e^{-\left( \frac{d}{2} + 1 \right)} } r^d = 
    	\frac{e}{\sqrt{2\pi}\sqrt{\frac{d}{2}+1} }
    	\left(\frac{\pi e}{\left( \frac{d}{2} + 1 \right)} \right)^{\frac{d}{2}} r^d.
    \end{align*}
    We deduce the bound on $\mathrm{Vol}(B(0,8r))$.
\end{proof}

\section{Proof of the main duality theorem}
\label{sec:proogTheoremClassicalDuality}
\begin{proof}[Proof of \Cref{th:classical_duality_theorem}]
    This is mostly classical, we provide arguments and literature pointers for completeness.
    \cite[Theorem 1.40]{book_santambrogio} ensures that there exists Kantorovich dual potentials optimal for the following problem
     \begin{equation}\label{eq:Th_dualite_eq1}
        \max_{(f,g)\in L^1(\mu)\times L^1(\nu)}\int_{\R^d}fd\mu+\int_{\R^d}gd\nu.
    \end{equation}
    Moreover the functions $(f,g)$ can be chosen such that $\varphi:=\frac{\|.\|^2}{2}-f$ and $\psi:=\frac{\|.\|^2}{2}-g$ verify $\varphi^*=\psi$ and $\psi^*=\varphi$ and thus $\varphi$ and $\psi$ are in $\Gamma_0(\R^d)$. Furthermore,  \cite[Theorem 1.40]{book_santambrogio} provides equality between \eqref{eq:Th_dualite_eq1} and \eqref{PP}. Therefore using this change of variables we obtain \eqref{PP} $=$ \eqref{DP} and both are being attained. 
    
    We now prove the equivalence between the three characterizations. The first equivalence $(1 \iff 2)$ is a consequence of \eqref{PP}$=$\eqref{DP}. 
    Let us show ($2 \iff 3$). Expanding the square in $\int_{\R^d\times\R^d}\frac{1}{2}\|x-y\|^2d\gamma(x,y)$ and using the fact that $\gamma$ has marginals $\mu$ and $\nu$ shows that point 2 is equivalent to  $\int_{\R^d\times\R^d}[\varphi(x)+\varphi^*(y)-\langle x,y \rangle ]d\gamma(x,y)=0$. By the Fenchel-Young inequality the integrand is positive, and using \Cref{prop:measure_theory} this is equivalent to $\supp \gamma \subset \{(x,y),\:\varphi(x)+\varphi^*(y)=\langle x,y \rangle\} = \graph(\partial\varphi)$.

    Let us now prove that an optimal potential in \eqref{DP} may be chosen Lipschitz. Let $\varphi$ be an optimal dual potential in \eqref{DP} and consider $\varphi_1=(\varphi^*+\delta(.|\bar{B}_{R_\nu}))^*$ . Classically (see \cite[Theorems 23.5 and 24.7]{rockafellar1970convex}), $\varphi_1$ is $R_\nu$-Lipschitz.
    Let us show that $\varphi_1$ is also optimal for \eqref{DP}. Fix $(x,y)\in\supp \gamma$, by the third characterization, we have $y\in\partial\varphi(x)$, hence $y\in\argmax_{z\in \R^d}\langle x,z \rangle-\varphi^*(z)$ by \cite[Theorem 23.5]{rockafellar1970convex}, so that
    \begin{equation*}
        \begin{split}
            \varphi(x)&=\varphi^{**}(x)= \sup_{z}\langle x,z \rangle-\varphi^*(z) = \langle x,y \rangle-\varphi^*(y) = \max_{\|z\|\leq R}\langle x,z \rangle-\varphi^*(z)
            =\varphi_1(x).
        \end{split}
    \end{equation*}
    Also we obviously have $\varphi_1^*(y)=\varphi^*(y)$ for all $y\in \supp \nu$. Since $\gamma$ is a coupling with marginals $\mu$ and $\nu$ the optimality of $\varphi_1$ is deduced from the previous arguments by integration: 
    \begin{equation*}
        \int_{\R^d\times \R^d}[\varphi_1(x)+\varphi_1^*(y)]  d\gamma(x,y)=\int_{\R^d\times\R^d} [\varphi(x)+\varphi^*(y)]d\gamma(x,y).
    \end{equation*}

    Suppose now that $\mu$ is absolutely continuous. The existence of a transport map  can be deduced from Rademacher's Theorem \cite[Theorem 1.17]{book_santambrogio}. Any Lipschitz optimal Brenier potential is differentiable almost everywhere such that $\partial \varphi(x) = \{\nabla \varphi(x)\}$ for $\mu$ almost all $x$. Consequently, since $\gamma$ has first marginal $\mu$ and is supported on $\graph \partial \varphi$, for any measurable function $f$ : $    \int f(x,y)\, d\gamma(x,y) \;=\; \int f\bigl(x,\nabla \varphi(x)\bigr)\, d\mu(x).$

    Uniqueness of the map can be shown using \cite[Remark 1.19]{book_santambrogio} for example. 
\end{proof}

\section{Proof of the finite-dimensional expression of the semi-discrete dual formula}\label{sec:semidual_semidiscrete}

\begin{proof}[Proof of \Cref{lem:semidiscreteSemidual}]
    For any $g \in \R^m$ and $\varphi_g \in \Gamma_0(\R^d)$, we have for all $j = 1,\ldots, m$
    \begin{align}
        \label{eq:semiDiscreteDualPotential}
        \varphi_g^*(y_j) &= \sup_{x \in \R^d} \left\langle y_j,x \right\rangle - \varphi_g(x) \nonumber\\
        &= \sup_{x \in \R^d} \left\langle y_j,x \right\rangle - \max_{k=1,\ldots,m} \left\{\left\langle x, y_k\right\rangle  + g_k-\frac{1}{2}\|y_k\|^2\right\}\nonumber\\
        &=\frac{1}{2}\|y_j\|^2 - g_j + \sup_{x \in \R^d} \left\langle y_j,x \right\rangle   + g_j-\frac{1}{2}\|y_j\|^2 - \max_{k=1,\ldots,m} \left\{\left\langle x, y_k\right\rangle  + g_k-\frac{1}{2}\|y_k\|^2\right\}\nonumber\\
        &\leq \frac{1}{2}\|y_j\|^2 - g_j.
    \end{align}
    Furthermore, if there is $x$ such that $\varphi_g(x) = \left\langle x, y_j\right\rangle + g_j - \frac{1}{2} \|y_j\|^2$, then \eqref{eq:semiDiscreteDualPotential} is actually an equality.

    Let us show that \eqref{eq:semiDiscreteSemiDual1} $=$ \eqref{eq:semiDiscreteSemiDual2}. \Cref{th:classical_duality_theorem} ensures that the infimum in \eqref{eq:semiDiscreteSemiDual1} is attained. Furthermore, we obviously have \eqref{eq:semiDiscreteSemiDual2} $ \geq$ \eqref{eq:semiDiscreteSemiDual1}.
    Consider an optimal $\varphi \in \Gamma_0(\R^d)$ for \eqref{eq:semiDiscreteSemiDual1}. It must hold that $\max_j \varphi^*(y_j) < + \infty$. For all $j$, set $g_j = \frac{1}{2} \|y_j\|^2 - \varphi^*(y_j)$ we have 
    \begin{align*}
        \varphi(x) &= \sup_y \left\langle x, y \right\rangle - \varphi^*(y) \geq \max_{k = 1,\ldots, m} \left\langle x, y_k \right\rangle - \varphi^*(y_k)  =\varphi_g(x) & \forall x \in \R^d, \nonumber\\
        \varphi^*(y_j) &= \frac{1}{2}\|y_j\|^2 - g_j \geq \varphi_g^*(y_j)& \forall j = 1,\ldots,m, 
    \end{align*}
    where we used \eqref{eq:semiDiscreteDualPotential}.
    We deduce that \eqref{eq:semiDiscreteSemiDual1} $\geq$ \eqref{eq:semiDiscreteSemiDual2}, so that both value are equal, and the infimum in \eqref{eq:semiDiscreteSemiDual2} is also attained.
    
    Let us now show that \eqref{eq:semiDiscreteSemiDual2} $=$ \eqref{eq:semiDiscreteSemiDual3}.
    First, \eqref{eq:semiDiscreteDualPotential} ensures that $\eqref{eq:semiDiscreteSemiDual2} \leq \eqref{eq:semiDiscreteSemiDual3}$. The proof of equality follows by the following claim, for example applied to an optimal $\bar{g}$ in \eqref{eq:semiDiscreteSemiDual2}.

    \begin{claim} 
        For any $\bar{g} \in \R^m$, there exists $g \geq \bar{g}$ coordinatewise such that $\varphi_g = \varphi_{\bar{g}}$ and \eqref{eq:semiDiscreteDualPotential} is an equality for $\varphi_g$ for all $j = 1,\ldots,m$. Furthermore, $g_j > \bar{g}_j$ for all $j$ such that \eqref{eq:semiDiscreteDualPotential} is not an equality for $\varphi_{\bar{g}}$.
    \end{claim}

    Let us now prove the claim. Set $\bar{J} = \{j = 1,\ldots,m|\, \exists x \in \R^d, \varphi_{\bar{g}}(x) = \left\langle x, y_j\right\rangle + \bar{g}_j - \frac{1}{2} \|y_j\|^2\}$. For all $j \in \bar{J}$, \eqref{eq:semiDiscreteDualPotential} is an equality and we may set $x_j$ such that $\varphi_{\bar{g}}(x_j) = \left\langle x_j, y_j\right\rangle + \bar{g}_j - \frac{1}{2} \|y_j\|^2$. Fix any $q \not \in \bar{J}$, it holds that \eqref{eq:semiDiscreteDualPotential} is strict, and actually
    \begin{align*}
        \min_x\, \varphi_{\bar{g}}(x)- \left\langle x, y_q\right\rangle - \bar{g}_q + \frac{1}{2} \|y_q\|^2 > 0,
    \end{align*}
    since the infimum of a polyhedral convex function is attained (see \cite[Chapter 19]{rockafellar1970convex}, the infimum is expressed as a bounded linear program whose value is always attained). We may denote the argmin by $\bar{x}$ and set 
    \begin{align*}
        g_q = \min_x\, \varphi_{\bar{g}}(x)- \left\langle x, y_q\right\rangle  + \frac{1}{2} \|y_q\|^2 = \varphi_{\bar{g}}(\bar{x})- \left\langle \bar{x}, y_q\right\rangle  + \frac{1}{2} \|y_q\|^2 > \bar{g}_q,
    \end{align*}
    and $g_k \colon = \bar{g}_k$ for $k \neq q$. We have $\varphi_{\bar{g}} \leq \varphi_g$ by construction, since $g \geq \bar{g}$ coordinatewise.
    
    For any $x$ such that $\varphi_g(x) > \langle x,y_q \rangle+g_q-\frac{1}{2}\|y_q\|^2$, we have 
    \begin{align*}
        \varphi_g(x) &= \max_{k \neq q}\langle x,y_k \rangle+g_k-\frac{1}{2}\|y_k\|^2 \\
        &= \max_{k \neq q}\langle x,y_k \rangle+\bar{g}_k-\frac{1}{2}\|y_k\|^2 \\
        &\leq \max_{k}\langle x,y_k \rangle+\bar{g}_k-\frac{1}{2}\|y_k\|^2\\
        &= \varphi_{\bar{g}}(x).
    \end{align*}
    For any $x$ such that $\varphi_g(x) = \langle x,y_q \rangle+g_q-\frac{1}{2}\|y_q\|^2$, we have
    \begin{align*}
        \varphi_g(x) &=\langle x,y_q \rangle+\min_z\, \varphi_{\bar{g}}(z)- \left\langle z, y_q\right\rangle  + \frac{1}{2} \|y_q\|^2 -\frac{1}{2}\|y_q\|^2
        \leq \varphi_{\bar{g}}(x).
    \end{align*}
    Thus $\varphi_g = \varphi_{\bar{g}}$. 
    
    Finally, we have 
    \begin{align*}
        \varphi_g(\bar{x})&= \max_{k=1,\ldots,m}\langle \bar{x},y_k \rangle + g_k-\frac{1}{2}\|y_k\|^2 \nonumber \geq \langle \bar{x},y_q \rangle + g_q-\frac{1}{2}\|y_q\|^2\nonumber\\
        &= \langle \bar{x},y_q \rangle+ \varphi_{\bar{g}}(\bar{x})- \left\langle \bar{x}, y_q\right\rangle  + \frac{1}{2} \|y_q\|^2-\frac{1}{2}\|y_q\|^2\nonumber 
        = \varphi_{\bar{g}}(\bar{x}) = \varphi_g(\bar{x}),
    \end{align*}
    and thus
    \begin{align}
        \label{eq:unionActiveIndex1}
        \varphi_g(\bar{x}) =  \langle \bar{x},y_q \rangle + g_q-\frac{1}{2}\|y_q\|^2.
    \end{align}
    In addition for all $j \in \bar{J}$
    \begin{align}
        \label{eq:unionActiveIndex2}
        \varphi_g(x_j) = \varphi_{\bar{g}}(x_j) =   \left\langle x_j, y_j\right\rangle + \bar{g}_j - \frac{1}{2} \|y_j\|^2 = \left\langle x_j, y_j\right\rangle + g_j - \frac{1}{2} \|y_j\|^2.
    \end{align}
    We deduce by combining \eqref{eq:unionActiveIndex1} and \eqref{eq:unionActiveIndex2} that  \eqref{eq:semiDiscreteDualPotential} holds with equality for $\varphi_g$, for all $j \in \bar{J} \cup \{q\}$. We may repeat this process recursively for each $q \not \in \bar{J}$ and obtain $\varphi_g = \varphi_{\bar{g}}$ such that \eqref{eq:semiDiscreteDualPotential} is an equality for all $j$. This shows \eqref{eq:semiDiscreteSemiDual2} $=$ \eqref{eq:semiDiscreteSemiDual3} and concludes the proof. 

   We now prove the last statement of the Lemma. Let $g\in \R^m$ be optimal in \eqref{eq:semiDiscreteSemiDual3} such that \eqref{eq:semiDiscreteDualPotential} is an equality for all $j$. The addition of a constant to all the coordinates of $g$ does not change the functional value in \eqref{eq:semiDiscreteSemiDual3}, we may therefore assume that $g_m=0$.  By \Cref{lem:brenier_potentials} there exists a function in $\mathcal{C}_{R,R}$ equal to $\varphi_g$ over $\supp \mu$ with conjugate equal to $\varphi_g^*$ over $\supp \nu$.
    For any $i\in \{1,\ldots,m\}$, using that $\varphi_g^*(y_i)=\frac{1}{2}\|y_i\|^2-g_i$ and the fact that $y_i,y_m\in \supp \nu$ we obtain
    \begin{align*}
        |g_i|&=|g_i-g_m|
        =\left|\frac{1}{2}\|y_i\|^2-\frac{1}{2}\|y_m\|^2-\varphi_g^*(y_i)+\varphi_g^*(y_m)\right| 
        \leq R^2+R\|y_i-y_m\|\leq 3R^2.
    \end{align*}
\end{proof}

\section{Technical proofs related to the convergence of couplings}
\label{sec:proof_couplings}
\begin{proof}[Proof of \Cref{lem:perturbCoulingSupport}]
    Fix any $x,y \in \supp\ \tilde{\gamma}$, $\lambda>0$, set $z_\lambda(x,y) = \prox_{\lambda\bar{\varphi}} (x+\lambda y)$ and $T(x,y) = (z_\lambda(x,y), y+ (x -z_\lambda(x,y))/\lambda)$, characterized as
    \begin{align}
    	y+\frac{x- z_\lambda(x,y)}{\lambda}  \in\partial \bar{\varphi} (z_\lambda(x,y)).
        \label{eq:proxTransportCoupling}
    \end{align}
    We have $\prox_{\lambda \bar{\varphi}}(x + \lambda v) = x$ for any $v \in \partial \bar{\varphi}(x)$, the $\prox$ mapping being $1$-Lipschitz, we deduce
    \begin{align*}
    	\|z_\lambda(x,y) - x \| \leq \lambda \| y -v \|\leq 2 \lambda L,
    \end{align*}
    where we have used Lipschitz continuity of $\tilde{\phi}$ and $\bar{\varphi}$ in a neighborhood of $x \in \supp \tilde \mu$, so that $y \in \partial \tilde{\phi}(x)$ and $v \in \partial \bar{\varphi}(x)$ have norm at most $L$.
    Now, choose $\lambda = R / (2L)$, we have $\|z_\lambda(x,y) - x \|\leq R$ and $z_\lambda(x,y) \in B(0,2R)$.

    Since $\tilde{\phi}$ is a Brenier potential for the pair $\tilde{\mu},\tilde{\nu}$, we have for $\tilde{\gamma}$ almost all $x,y$, $\tilde{\phi}(x) + \tilde{\phi}^*(y) = \left\langle x, y \right\rangle$. Using Carlier's inequality in \Cref{th:Carlier} and the definition of $z_\lambda$, we obtain
    \begin{align}
    	&\int \bar{\varphi} - \tilde{\phi} d\tilde{\mu} + \int \bar{\varphi}^* - \tilde{\phi}^* d\tilde{\nu} \\
    	& =\int \bar{\varphi}(x) + \bar{\varphi}^*(y)- \left\langle x, y \right\rangle d \tilde{\gamma}(x,y) \geq\int \frac{1}{\lambda}\|x-z_\lambda (x,y)\|^2 d\tilde{\gamma}(x,y)
        \label{eq:lowerBoundCarlier}
    \end{align}
    Set $\gamma_0 = T_\sharp \tilde{\gamma}$ and $\mu_0$ its first marginal. We have that $\supp\, \gamma_0 \subset  \graph \ \partial \bar{\varphi}$ by construction, see \eqref{eq:proxTransportCoupling} and $\mathrm{Supp}\ \mu_0 \subset B(0,2R)$ by the argument above. We also have $\|T(x,y) - (x,y)\|^2 = (1 + 1 / \lambda^2) \|x - z(x,y)\|^2$ so that
    \begin{align*}
    	W_2^2(\tilde{\gamma}, \gamma_0) &\leq  \int \|T(x,y) - (x,y)\|^2 d\tilde{\gamma}(x,y) = \left( 1 + \frac{1}{\lambda^2} \right)\int \|x - z_\lambda(x,y)\|^2 d\tilde{\gamma}(x,y)\\
    	&\leq \left( \lambda + \frac{1}{\lambda} \right) \left(\int \bar{\varphi} - \tilde{\phi} d\tilde{\mu} + \int \bar{\varphi}^* - \tilde{\phi}^* d\tilde{\nu}\right),
    \end{align*}
    where the last inequality is from \eqref{eq:lowerBoundCarlier}.
    This is our first inequality.
    
    Let us proceed to the second inequality. Denoting by $\pi_x$ the projection on the $x$ coordinate, we have that $(\pi_x, z_\lambda)_\sharp \tilde{\gamma}$ has first marginal $\tilde{\mu}$ and second marginal $\mu_0$. Indeed, $\mu_0$ is the first marginal of $\gamma_0$, i.e. $\mu_0=\pi_{x\sharp}\gamma_0 = \pi_{x\sharp} (T_\sharp \tilde{\gamma})=(\pi_x\circ T)_\sharp\tilde{\gamma}=z_{\lambda\sharp}\tilde{\gamma}$. Similarly, $\pi_{y\sharp}[(\pi_x,z_\lambda)_\sharp\tilde{\gamma}] = [\pi_{y}\circ(\pi_x,z_\lambda)]_\sharp\tilde{\gamma}=z_{\lambda\sharp}\tilde{\gamma}$ is the second marginal of $(\pi_x, z_\lambda)_\sharp \tilde{\gamma}$. Thus, $(\pi_x, z_\lambda)_\sharp \tilde{\gamma}$ defines a valid coupling for $\tilde{\mu}$ and $\mu_0$ and we deduce our second inequality
    \begin{align*}
    	W_2^2(\tilde{\mu}, \mu_0) & \leq \int \|x-z_\lambda (x,y)\|^2 d\tilde{\gamma}(x,y) \leq \lambda \left(\int \bar{\varphi} - \tilde{\phi} d\tilde{\mu} + \int \bar{\varphi}^* - \tilde{\phi}^* d\tilde{\nu} \right)
    \end{align*}
\end{proof}   

\begin{proof}[Proof of \Cref{lem:extensionStability}]
	Denote by $S$ the optimal mapping from $\mu$ to $\beta$, see \Cref{th:classical_duality_theorem}.
    Define $\gamma_x$ the disintegration of $\gamma$, such that for $\mu$ almost all $x$, $\supp\ \gamma_x \subset \partial \bar{\varphi}(x) = \{\nabla \bar{\varphi}(x)\}$, and for any measurable function on $h \colon \R^d \times \R^d \to \R$, $\int h(x,y) d \gamma(x,y) = \int h(x,y) d\gamma_x(y) d\mu(x)$. Similarly, $\theta_z$ is the disintegration of $\theta$
    such that $\supp\ \theta_z \in \partial \bar{\varphi}\left( z \right)$ for all $z$ and for any measurable function on $h \colon \R^d \times \R^d \to \R$, $\int h(x,y) d \theta(x,y) = \int h(x,y) d\theta_x(y) d\beta(x)$.
    
	Consider a coupling $\Gamma$, between $\gamma$ and $\theta$ as follows, the marginal in $x$ is $\mu$ and the remaining variables are independent conditioned on $x$, with distribution
	\begin{align*}
		y,\tilde{x}, \tilde{y} | x \sim \gamma_x, S(x), \theta_{S(x)}.
	\end{align*}
	We have $S_\sharp \mu = \beta$ so that the second marginal of $\Gamma$ is $\theta$. The first marginal is obviously $\gamma$. The associated quadratic cost provides an upper bound on the $W_2$ distance between the couplings, 
	\begin{align*}
		W_2^2(\gamma, \theta) \leq\,&\int \|x - \tilde{x}\|^2 + \|y - \tilde{y}\|^2 d\Gamma(x,y,\tilde{x},\tilde{y})\\
        &=\int \|x - S(x)\|^2 + \|\nabla \bar\varphi(x) - \tilde{y}\|^2 d\theta_{S(x)} (\tilde{y}) d\mu(x)
	\end{align*}
    Choose $\eta \in ]0,R[$ and $\alpha \geq 0$. Set $\Omega_\eta = \left\{ x| \|x - S(x)\| \geq \eta \right\}$, by Markov's inequality, we have
	\begin{align*}
		\mu(\Omega_\eta) \leq \frac{1}{\eta^2} \int \|x - S(x)\|^2 d\mu(x) = \frac{1}{\eta^2}  W_2^2(\mu,\beta).
	\end{align*}
    We deduce that 
    \begin{align*}
        W_2^2(\gamma, \theta) \leq \left(1 + \frac{4L^2}{\eta^2}\right) W_2^2(\mu,\beta) + \int_{x \not \in \Omega_\eta} \|\nabla \bar\varphi(x) - \tilde{y}\|^2 d\theta_{S(x)} (\tilde{y}) d\mu(x)
    \end{align*}
    where we have used that $\|\nabla \bar{\varphi}(x)\| \leq L$ for $\mu$ almost all $x$, and $\theta$ is supported on its graph so that $\|\tilde{y}\| \leq L$ for $\mu$ almost all $x$ and $\theta_{S(x)}$ almost all $\tilde{y}$, by Lipschitz continuity of $\bar{\varphi}.$
    
	Set $\Sigma_{\eta,\alpha} = \left\{ x \not\in \Omega_\eta| \mathrm{diam}(\partial \bar{\varphi}(B(x,\eta))) \geq \alpha \right\}$. Note that for all $x\in B_R$: $B(x,\eta)\subset B(0,2R)$ We have, using similar bounds on $\nabla \bar\varphi$ and $\tilde y$:
	\begin{align}
		\label{eq:tempIneqCoupling}
		\int_{x \not\in\Omega_\eta} \|\nabla \bar\varphi(x) - \tilde{y}\|^2 d\theta_{S(x)} (\tilde{y}) d\mu(x)&\leq 4L^2 \mu(\Sigma_{\eta,\alpha}) + \alpha^2 \nonumber\\
		&\leq 4L^2M \ \mathrm{Vol}(B(0,8\eta)) \ \mathcal{N}(\Sigma_{\eta,\alpha} \cap B(0,R), 8 \eta)  + \alpha^2.
	\end{align}
	We get for any $\eta >0$, $\alpha \geq 0$
	\begin{align}
		\label{eq:mainIneqCoupling}
		W_2^2(\gamma, \theta) \leq\left(1 + \frac{4L^2}{\eta^2}\right)W_2^2(\mu,\beta) + 4L^2M \ \mathrm{Vol}(B(0,8\eta)) \ \mathcal{N}(\Sigma_{\eta,\alpha} \cap B(0,R), 8 \eta) + \alpha^2
	\end{align}
    We are now going to treat two cases and choose corresponding values for $\eta$ and $\alpha$, note that we do not optimize constants.
	First consider the case $\left(\frac{6W_2^2(\mu,\beta)}{(2MC_{R,d})^{2/3}}\right)^{\frac{3}{8}} > \frac{R}{8}$. We choose $\eta = \frac{R}{8}$, $\alpha = 0$, and obtain using the volume estimates in \Cref{Lem:volume_estimate_ball} and the definition of $C_{R,d}$ from \Cref{th:dualErrorBound},
	\begin{align*}
		W_2^2(\gamma, \theta) &\leq \left( 1 + \frac{256L^2}{R^2}\right)W_2^2(\mu,\beta) + 4L^2M \mathrm{Vol}(B(0,8\eta)) \\
		&\leq \left( 1 + \frac{256L^2}{R^2}\right)W_2^2(\mu,\beta) + \frac{1}{8} L^2M C_{R,d} R\\
		&\leq \left( 1 + \frac{256L^2}{R^2}\right)W_2^2(\mu,\beta) + L^2M C_{R,d}\left(\frac{6W_2^2(\mu,\beta)}{(2C_{R,d})^{2/3}}\right)^{\frac{3}{8}}.
	\end{align*}

	Second consider the case $\left(\frac{6W_2^2(\mu,\beta)}{(2MC_{R,d})^{2/3}}\right)^{\frac{3}{8}} \leq \frac{R}{8}$.
	We get from \eqref{eq:mainIneqCoupling} and \Cref{th:coveringNumber}, for any $\eta \leq R/4$ and $\alpha \geq 0$,
	\begin{align*}
		W_2^2(\gamma, \theta) \leq \left(1 + \frac{4L^2}{\eta^2}\right)W_2^2(\mu,\beta) + \frac{L^3M C_{R,d}\eta}{\alpha} + \alpha^2.
	\end{align*}
	Choosing $\alpha = (L^3MC_{r,d}\eta)^{\frac{1}{3}}$, we obtain
	\begin{align*}
		W_2^2(\gamma, \theta) \leq \left(1 + \frac{4L^2}{\eta^2}\right)W_2^2(\mu,\beta) + 2(L^3MC_{r,d}\eta)^{\frac{2}{3}}. 
	\end{align*}
	Choosing $\eta = \left(6\frac{W_2^2(\mu,\beta)}{(2MC_{r,d})^{2/3}}\right)^{\frac{3}{8}} \leq \frac{R}{8}$, we obtain 
	\begin{align*}
		&W_2^2(\gamma, \theta) \leq \left(1 + \frac{4L^2}{\eta^2}\right)W_2^2(\mu,\beta) + 2(L^3MC_{r,d}\eta)^{\frac{2}{3}} \\
		=\,&W_2^2(\mu,\beta) + 4L^2\left(\frac{1}{6}(2MC_{r,d})^{2/3}\right)^{\frac{3}{4}} W_2^2(\mu,\beta)^{\frac{1}{4}}\\
        &+ 2(L^3MC_{r,d})^{\frac{2}{3}}  \left(\frac{6}{(2MC_{r,d})^{2/3}}\right)^{\frac{1}{4}} W_2^2(\mu,\beta)^{\frac{1}{4}} \\
        =\,& W_2^2(\mu,\beta) + W_2^2(\mu,\beta)^{\frac{1}{4}} L^2(2MC_{r,d})^{\frac{1}{2}} \left(4\cdot 6^{-3/4} + 2 \cdot 6^{1/4}\right)\\
		\leq\,& W_2^2(\mu,\beta) +  5 L^2 {\color{black}(2MC_{r,d})}^{\frac{1}{2}}  W_2^2(\mu,\beta)^{\frac{1}{4}} .
	\end{align*}
\end{proof}

\begin{proof}[Proof of \Cref{cor:stabilityCoupling}]
    By \Cref{lem:brenier_potentials} there is an optimal Brenier potential $\bar{\varphi} \in \mathcal{C}_{4R_\mu,4R_\nu}$, for measures $\mu,\nu$. The same is true for measures $\tilde{\mu}$,  $\tilde{\nu}$ and Brenier potential $\tilde{\phi} \in \mathcal{C}_{4R_\mu,4R_\nu}$. The hypotheses of \Cref{th:stabilityCoupling} are satisfied for $\bar{\varphi}$ and $\tilde{\phi}$.  Furthermore using \Cref{th:classical_duality_theorem} for measures $\mu$ and $\nu$,  we have
    \begin{equation*}
        \int_{\R^d}(\bar{\varphi}-\tilde{\phi})d\mu+\int_{\R^d}(\bar{\varphi}^*-\tilde{\phi}^*)d\nu\leq 0.
    \end{equation*}
    From \Cref{th:stabilityCoupling} there is $C>0$, such that for any $\mu,\nu,\tilde{\mu},\tilde{\nu}$ as in the statement, we have
    \begin{align}
        \frac{1}{C^4} W_2(\gamma,\tilde{\gamma})^4 &\leq\, W_2(\mu,\tilde{\mu}) +  \left(\int \bar{\varphi} - \tilde{\phi} d\tilde{\mu} + \int \bar{\varphi}^* - \tilde{\phi}^* d\tilde{\nu} \right)^{\frac{1}{2}}\nonumber\\
       &\leq\, W_2(\mu,\tilde{\mu}) +  \left(\int \bar{\varphi} - \tilde{\phi} d(\tilde{\mu} - \mu) + \int \bar{\varphi}^* - \tilde{\phi}^* d(\tilde{\nu} - \nu) \right)^{\frac{1}{2}} 
       \label{eq:cor_stability_coupling1}
    \end{align}
    In view of applying \cite[Theorem 1.14]{ref_villani_topics}, we remark that $\bar{\varphi}-\tilde{\phi}$ and $\bar{\varphi}^*-\tilde{\phi}^*$ are $8R_\nu$ and $8R_\mu$ Lipschitz respectively. Therefore
    \begin{align}
        \int \bar{\varphi} - \tilde{\phi} d(\tilde{\mu} - \mu) + \int \bar{\varphi}^* - \tilde{\phi}^* d(\tilde{\nu} - \nu)& \leq 8R_\nu W_1(\tilde{\mu},\mu)+8R_\mu W_1(\tilde{\nu},\nu) \nonumber \\
        &\leq 8\max(R_\mu,R_\nu)(W_2(\tilde{\mu},\mu)+W_2(\tilde{\nu},\nu)).
        \label{eq:cor_stability_coupling2}
    \end{align}
    Combining \eqref{eq:cor_stability_coupling1} and \eqref{eq:cor_stability_coupling2} we obtain
    \begin{align*}
        \frac{1}{C^4}W_2(\gamma,\tilde{\gamma})^4
        &\leq W_2(\mu,\tilde{\mu})+3\sqrt{\max(R_\mu,R_\nu)}(W_2(\tilde{\mu},\mu)+W_2(\tilde{\nu},\nu))^{\frac{1}{2}}\\
        &\leq W_2(\mu,\tilde{\mu})+3\sqrt{\max(R_\mu,R_\nu)}\left(W_2(\tilde{\mu},\mu)^{\frac{1}{2}}+W_2(\tilde{\nu},\nu)^{\frac{1}{2}}\right).
    \end{align*}
    With a similar reasoning as the one made in the proof of \Cref{cor:simplerForm} and \Cref{th:stabilityCoupling} we can adjust the constants in the previous equation so that only the smallest powers dominate. This concludes the proof.
\end{proof}

\section{Proof illustration}
\Cref{fig:illustrProofStability} illustrates the proof of \Cref{th:stabilityCoupling}.
\begin{figure}[H]
    \centering
    \includegraphics[width=0.6\linewidth]{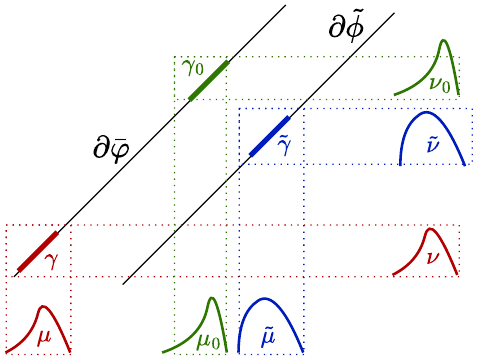}
    \caption{A visual illustration of the proof of \Cref{th:stabilityCoupling}. $\gamma$ and $\tilde{\gamma}$ are supported on the graph of $\partial \bar{\varphi}$ and $\partial \tilde{\phi}$ respectively. \Cref{lem:perturbCoulingSupport} provides $\gamma_0$, on the graph of $\partial \bar{\varphi}$, such that $W_2(\tilde{\gamma}, \gamma_0)$ is controlled by a suboptimality gap. Since $\gamma$ and $\gamma_0$ are both supported on the graph of $\partial \bar{\varphi}$, \Cref{lem:extensionStability} allows to control $W_2(\gamma, \gamma_0)$ with a $W_2(\mu, \mu_0)$ term, in turn controlled by $W_2(\mu, \tilde{\mu})$ and the suboptimality gap in \Cref{lem:perturbCoulingSupport}.}
    \label{fig:illustrProofStability}
\end{figure}

\section{On obtaining improved rates}
\label{sec:improvedRates}
As described in the conclusion of the main text, our statistical convergence rates could be improved, see also \cite{merigot2026sharp}. The main argument is as follows, it relies on \cite[Corollary 2.2]{carlier2025quantitative} instead of \Cref{th:coveringNumber}. In the proof of \Cref{th:dualErrorBound}, equation \eqref{eq:ineqTemp2} could be replaced by the following, for $q>1$:

\begin{align}
	\int \| v_\lambda - \nabla \phi\|^q d\mu \leq\,& \int \mathrm{diam}(\partial \phi(B(x,\eta)))^q d\mu(x) \nonumber\\
    &\leq M\int \mathrm{diam}(\partial \phi(B(x,\eta)))^q d x \nonumber \leq ML^qC \eta 
\end{align}
for some constant $C$ given in \cite[Corollary 2.2]{carlier2025quantitative}. The obtained bound can then be combined with a computation similar to \Cref{lem:technicalComputation} to obtain case 2 in \Cref{th:dualErrorBound}, with an exponent $q+2$ instead of $q+3$ for $q>1$.

Similarly in the proof of \Cref{lem:extensionStability}, equation \eqref{eq:tempIneqCoupling} could be replaced by
\begin{align}
	\label{eq:tempIneqCouplingBis}
	\int_{x \not\in\Omega_\eta} \|\nabla \bar\varphi(x) - \tilde{y}\|^2 d\theta_{S(x)} (\tilde{y}) d\mu(x)& \leq M\int \mathrm{diam}(\partial \phi(B(x,\eta)))^2 d x \nonumber \leq ML^2C \eta 
\end{align}
for some constant $C$ given in \cite[Corollary 2.2]{carlier2025quantitative}. Similar computations lead to an inequality of the form
\begin{align*}
	W_2^2(\gamma,\theta) \leq \left( 1 + \frac{4L^2}{\eta^2} \right) W_2^2(\mu,\beta) + c \eta
\end{align*}
for some constant $c$. Optimizing over $\eta$ leads to an exponent $1/3$ instead of $1/4$ in case two of \Cref{lem:extensionStability}. The same reasoning as in \Cref{th:stabilityCoupling} leads to an exponent $1/6$ instead of $1/8$. 

\section{Measure theoretical lemmas}
The following are classical, we provide proof arguments for completeness.
\begin{definition}
 The support of a probability measure $\mu$ on $\R^d$  is the smallest closed set $X\subset\R^d$ such that  $\mu(X)=1$ . That is:
 \begin{equation*}
    \supp \mu=\bigcap\left\{X\subset\R^d\:|\:X\:\text{closed},\:\mu(X)=1\right\}.
 \end{equation*}
 Or equivalently: 
 \begin{equation*}
    \supp \mu= \{x\in \R^d\:|\:\mu(B(x,\epsilon))>0\:\forall\epsilon>0\}.
 \end{equation*}
\end{definition}
\begin{definition}
For $i=1,2$ we denote $\pi_{i}:\R^d\times \R^d\longrightarrow\R^d$ the projection mapping on the first and respectively second component, \textit{i.e} for every $(x_1,x_2)\in\R^d\times \R^d$: $\pi_i(x_1,x_2)=x_i$.
\end{definition}
\begin{proposition}
    Let $\mu$ be a probability measure and $f$ be a non-negative lower semi-continuous function. Then $\int_{\R^d}fd\mu=0$ if and only if $\supp \mu\subset \{x\in\R^d,\:f(x)=0\}$.
    \label{prop:measure_theory}
\end{proposition}
\begin{proof}
    Let $x\in\supp \mu$. For all $\epsilon>0$ we always have $\int_{B(x,\epsilon)}fd\mu=0$. On the other hand we also have $\int_{B(x,\epsilon)}fd\mu\geq \inf_{t\in B(x,\epsilon)}f(t)\mu(B(x,\epsilon))$. By definition of the support $\mu(B(x,\epsilon))>0$ for all $\epsilon>0$. It results that for any $x\in \supp \mu$ and for any $\epsilon>0,$ $\inf_{t\in B(x,\epsilon)}f(t)=0$. Taking the limit as $\epsilon\longrightarrow0$ we thus obtain $\lim_{\epsilon\rightarrow 0}\inf_{t\in B(x,\epsilon)}f(t)=\lim\inf_{t\rightarrow x}f(t)\geq f(x)$ by lower semi-continuity. This shows that $f(x)=0$.
    \end{proof}

\begin{lemma}
    Let $\mu,\:\nu$ be be compactly supported probability measures on $\R^d$ and $\gamma\in \Pi(\mu,\nu)$. Then 
    
    \begin{equation*}
        \supp \mu=\{x\in \R^d,\:\exists y\in \supp \nu\:\text{s.t}\:(x,y)\in\supp(\gamma)\}.
    \end{equation*}
    \label{lem:projection_support_coupling}
\end{lemma}
\begin{proof}
    By \cite[Lemma 3]{mccann2010optimal} we have:
    \begin{equation}\label{eq:proj1}
        \supp \mu=\overline{\mathrm{proj}_1(\supp \gamma)}
    \end{equation}
    where $\mathrm{proj}_i:\R^d\times \R^d\longrightarrow\R^d$ denotes the canonical projection mapping on the $i$'th component. We thus deduce the following sequence of inclusions. 
    \begin{align*}
        \supp \gamma \ &\subset \mathrm{proj}_1(\supp \gamma) \times \mathrm{proj}_2(\supp \gamma)\\
        &\subset \overline{\mathrm{proj}_1(\supp \gamma)} \times \overline{\mathrm{proj}_2(\supp \gamma)} =\supp \mu \times \supp \nu.
    \end{align*}
    We deduce that $\supp \gamma$ is compact, and for any $(x,y)\in \supp \gamma$, $y\in \supp \nu$. It follows that
    \begin{equation}\label{eq:proj2}
        \begin{split}
            \mathrm{proj}_1(\supp \gamma)&:=\left\{x\in \R^d,\:\exists y\in \R^d,\:(x,y)\in \supp \gamma \right\}\\
            &=\left\{x\in \R^d,\:\exists y\in \supp \nu,\:(x,y)\in \supp \gamma \right\}. 
        \end{split}
    \end{equation}
    It remains to show that $\mathrm{proj}_1(\supp \gamma)$ is closed, which follows from boundedness. Let $(x_n)_{n\in \mathbb{N}^*}\subset \mathrm{proj}_1(\supp \gamma)$ be a convergent sequence. For every $n\in \mathbb{N}^*$, there exists $y_n\in \supp \nu$ such that $(x_n,y_n)\in \supp \gamma$. Since $\supp \gamma$ is compact, there exists a subsequence of $(x_n,y_n)_n$ converging to some $(x^*,y^*)\in \supp \gamma$. Since $(x_n)_n$ converges, we necessarily have $x^*\in \mathrm{proj}_1(\supp \gamma)$. Hence, $\mathrm{proj}_1(\supp \gamma)$ is closed. Combining \eqref{eq:proj1} and \eqref{eq:proj2} thus yields the result.
\end{proof}

\section{Computational lemma}
\label{sec:LemmaPowers}

\begin{lemma}
    Assume $\delta,L,M,C > 0$ and $q \in [1,2]$.
	For $\eta = \left(\frac{\sqrt{\delta}}{\sqrt{2L} (M C)^{\frac{1}{q+1}}}\right)^{\frac{2q + 2}{q+3}}$, we have
	\begin{align*}
		2 L( M C\eta)^{\frac{1}{q+1}} + \sqrt{\frac{2L\delta}{\eta}} = 2 (2L)^{\frac{q+2}{q+3}} (MC\delta)^{\frac{1}{q+3}} 
	\end{align*}
\end{lemma}

\begin{proof}
	We compute powers, we have $2 L( M C\eta)^{\frac{1}{q+1}} = (2L)^{a} (MC)^b \delta^c$ with
	\begin{align*}
		a &= 1 -  \frac{2q+2}{2(q+1)(q+3)}= \frac{q+2}{q+3}  \\
		b &= \frac{1}{q+1}\left( 1 - \frac{2q+2}{(q+1)(q+3)} \right) = \frac{1}{q+3} & c = \frac{2q +2 }{2 (q+1) (q+3)} = \frac{1}{q+3}.
	\end{align*}
	Similarly, we have $\sqrt{\frac{2L\delta}{\eta}} = (2L)^{a} (MC)^b \delta^c$ with
	\begin{align*}
		a &= \frac{1}{2} + \frac{2q+2}{4(q+3)} = \frac{q+3 + q+1 }{2(q+3)} = \frac{q+2}{q+3} \\
		b &= \frac{2q+2}{2(q+1)(q+3)} = \frac{1}{q+3}  
		&c = \frac{1}{2} - \frac{2q+2}{4(q+3)} = \frac{q+3 - (q+1)}{2(q+3)} = \frac{1}{q+3} 
	\end{align*}
\end{proof}

\bibliographystyle{apalike}
\bibliography{ref}

\end{document}